\documentclass[pdflatex]{amsart}

 \usepackage{amsmath,times} 
\usepackage{amssymb}
 \usepackage{latexsym} \usepackage{amscd}

\usepackage[latin1]{inputenc} \newtheorem{theorem}{Theorem}[section]
\newtheorem{lemma}[theorem]{Lemma}
\newtheorem{proposition}[theorem]{Proposition}
\newtheorem{corollary}[theorem]{Corollary}
 \newtheorem{definition}[theorem]{Definition}
\newtheorem{example}[theorem]{Example}
\newtheorem{remark}[theorem]{Remark}
\newtheorem{conjecture}[theorem]{Conjecture}

\begin{document}

\title{Envelopes of positive metrics with prescribed singularities} \author{Julius
  Ross and David Witt Nystr\"om}

\begin{abstract}
We investigate envelopes of positive metrics with a prescribed singularity type.  First we generalise work of Berman to this setting, proving $C^{1,1}$ regularity of such envelopes, showing their Monge-Amp\`ere measure is supported on a certain ``equilibrium set'' and connecting with the asymptotics of the partial Bergman functions coming from multiplier ideals.  We investigate how these envelopes behave on certain products, and how they relate to the Legendre transform of a test curve of singularity types in the context of geodesic rays in the space of K\"ahler potentials.    Finally we consider the associated exhaustion function of these equilibrium sets, connecting it both to the Legendre transform and to the geometry of the Okounkov body. \medskip

\end{abstract}
\maketitle

\newcommand{\PSH}{PSH} 
\newcommand{\Sing}{\operatorname{Sing}}
\newcommand{\vol}{\operatorname{vol}}
\newcommand{\Conv}{\operatorname{Conv}}
\newcommand{\ord}{\operatorname{ord}}
\newcommand{\im}{\operatorname{im}}

\setlength{\parskip}{2pt}

\section{Introduction}

In this paper we study a number of features of envelopes of positive metrics with prescribed singularities.    The setting we shall consider consists of a compact complex manifold $X$ with a locally bounded metric $\phi$ on a line bundle $L$ and a positive singular metric $\psi$ on an auxiliary line bundle $F$.  The \emph{maximal envelope} of this data is defined to be
\[\phi_{[\psi]} = \sup \{ \gamma\in \PSH(L) : \gamma\le \phi \text{ and } \gamma 
\le \psi + O(1) \}^*\]
where the notation means the upper semicontinuous regularisation of the supremum of all positive metrics $\gamma$ on $L$ that are bounded by $\phi$ and have the same singularity type as $\psi$.   This maximal envelope is itself a positive metric on $L$ which, as the notation suggests, depends on $\phi$ and the singularity type of $\psi$.   

Before turning to precise statements, we begin with an overview of the contents of this paper which starts with some general statements about these envelopes, and then moves on to a number of applications and special cases.  The main technical result is Theorem \ref{thm:regularityintroduction} in which we prove, under suitable hypothesis, that $\phi_{[\psi]}$ is $C^{1,1}$ away from an obvious singular locus.  The reader may choose on first reading to take this statement as given, at which point the applications to the Monge-Amp\`ere measures (Theorem \ref{thm:extremalmeasures}) and to the partial Bergman function (Theorems \ref{thm:partialbergman} and \ref{thm:weakmeasureconvergence}) follow rather easily.  The proof of these statements form the largest part of the paper, and take up Sections 2 through to Section 4.  

Following this, in Section \ref{sec:products} we consider the case of maximal envelopes on products, which we show, at least in the algebraic case, is related to the Musta\c{t}\u{a} summation formula for multiplier ideals.  This is independent of the above technical results, and relies only on the statement in Theorem \ref{thm:partialbergman} concerning the partial Bergman function.    

In Section \ref{sec:legendre} we show how a previous construction of the authors \cite{RossNystrom} of geodesics in the space of K\"ahler potentials has an interpretation as a maximal envelope (on a product).  This too is independent of the above technical results, but the real interest is that they give, as a corollary, certain regularity of these geodesic rays (Theorem \ref{thm:reggeo}).  This topic is continued in Section \ref{gradientmaps} in which this regularity allows us to interpret the time derivative of this associated geodesic as a certain ``exhaustion function" that appears naturally from the definition of maximal envelopes.  Finally we consider in Section \ref{sec:divisorial} a very special case of this construction, essentially given by a singularity type along a divisor, and show how this exhaustion map gives a natural candidate for the ``first cooordinate" of a kind of moment-map from a polarised manifold to its Okounkov body.

\subsection{Regularity, Monge Amp\`ere measures and Partial Bergman Kernels} Our first set of results generalise work of Berman \cite{Berman} to this setting, and for convenience we collect the precise statements here.  In the following $\psi\in \PSH(F)$ will assumed to be exponentially H\"older continuous (which holds, for instance, if $\psi$ has analytic singularities).   

\begin{theorem}\label{thm:regularityintroduction}
  If $\phi$ is Lipschitz (resp.\ in the class $C^{1,1}$) over $X$ then the same is true of $\phi_{[\psi]}$ over $X-\mathbb B_+(L-F)\cup \Sing(\psi)$.
\end{theorem}
Here $\mathbb B_+(L-F)$ denotes the augmented base locus of $L-F$, and $\Sing(\psi)$ is the locus on which $\psi$ is not locally bounded.   We may as well assume that $L-F$ is big since otherwise $\mathbb B_+(L-F)=X$ and the statement becomes vacuous.   This result is in some sense optimal, since even in the case without the singular metric there are examples of maximal envelopes that are no more than $C^{1,1}$. \medskip

Assume now that $\phi$ is in fact $C^2$, denote by $U$ the set on which $\phi_{[\psi]}$ is locally bounded and set $X(0) =\{ x: dd^c \phi_x>0\}$.  The \emph{equilibrium measure} is defined as
\[\mu(\phi,\psi) := \mathbf{1}_{U} MA(\phi_{[\psi]}) = \frac{1}{n!}\mathbf 1_{U} dd^c (\phi_{[\psi]})^n\]
where ${\mathbf 1}_U$ is the characteristic function of $U$, and the \emph{equilibrium set} is
\[D=D(\phi,\psi) = \{x\in X : \phi_{[\psi]}(x)=\phi(x)\}.\]

\begin{theorem}\label{thm:extremalmeasures}
Assume $L-F$ is big.  Then there is an equality of measures
\[ \mu(\phi,\psi) = \mathbf 1_{X-\mathbb B_+(L-F)\cup \Sing(\psi)} MA(\phi_{[\psi]}) = \mathbf 1_{D} MA(\phi) = \mathbf 1_{D\cap X(0)} MA(\phi).\]
\end{theorem}

We remark that this theorem justifies the terminology, since $MA(\phi_{[\psi]})=0$ away from $D$ and thus $\phi_{[\psi]}$ is (locally) maximal there.\bigskip

Turning to algebraic data, denote the multiplier ideal sheaf of $k\psi$ by $\mathcal I(k\psi)$, and define the \emph{partial Bergman function} of $\phi$ and $\psi$ to be the smooth function on $X$ given by
\[B_k(\phi,\psi) = \sum_{\alpha} |s_{\alpha}|_{\phi}^2,\]
where $\{s_{\alpha}\}$ is any basis for $H^0(\mathcal I(k\psi) \otimes L^k)$ that is orthonormal with respect to the $L^2$-norm induced by $\phi$ and some fixed smooth volume form $dV$. 

\begin{theorem}\label{thm:partialbergman}
There is a limit
    \[ k^{-1} \ln B_k(\phi,\psi) \to \phi_{[\psi]}-\phi\] 
as $k$ tends to infinity that holds uniformly on compact subsets of $X-\mathbb B_+(L-F) \cup \Sing(\psi)$.

  More precisely, for each such compact set $K$ there is a $C_K>0$ such that for all $k$
    \[ C_K^{-1} e^{-k (\phi-\phi_{[\psi]})} \le B_k(\phi,\psi) \le C_K
    k^n e^{-k (\phi-\phi_{[\psi]})}. \] 
over $K$.
\end{theorem} 

\begin{theorem}\label{thm:weakmeasureconvergence}
Suppose $L-F$ is big.   Then there is a pointwise limit 
\[ \lim_{k\to \infty} k^{-n} B_k(\phi,\psi) dV = \mathbf 1_{D(\phi,\psi)\cap X(0)} MA(\phi)\]
almost everywhere on $X(0)$.  Moreover \[\lim_{k\to \infty} k^{-n} B_k(\phi,\psi)dV \to \mu(\phi,\psi)\] weakly in the sense of measures.
\end{theorem}

\subsection{Maximal Envelopes on Products}

Following these technical results we turn to maximal envelopes on products.  Suppose that we have two sets of data of the above kind, given by $(X_i,L_i,F_i,\phi_i,\psi_i)$  for $i=1,2$ where $L_i$ is a line bundle on a compact complex manifold $X_i$, $\phi_i$ a smooth metric on $L_i$ and $\psi_i$ a positive singular metric on $F_i$.  For simplicity assume $L_i-F_i$ is ample and $\psi_i$ has algebraic singularities for $i=1,2$.

\begin{theorem}
Consider the product metric $\phi=\phi_1 + \phi_2$ on $L_1\otimes L_2$ (where we suppress the pullback notation), and let
\[\psi = \sup \{ \psi_1, \psi_2\}.\]
Then
  \[ \phi_{[\psi]} = \sup\{ (\phi_1)_{[\lambda\psi_1]} + (\phi_2)_{[(1-\lambda)\psi_2]} : \lambda\in (0,1) \}^*.\]
\end{theorem}

This result resembles known formulae for the Siciak extremal function \cite{Bedford4, Blocki, Sickiak} and for the pluricomplex Green function \cite{Rashkovskii} on products.     The particular proof we give uses the connection with partial Bergman functions and gives an interesting interplay between this circle of ideas and the the Musta\c{t}\u{a} summation formula for multiplier ideals.  We do not suggest that the previous Theorem is optimal, and discuss conjectural generalisations in Section \ref{sec:products}.  However rather than pursing this we move on to consider other aspects of maximal envelopes that can be thought of as a special case in which $X_2$ is the unit disc in $\mathbb C$ and $\psi_2$ has a logarithmic singularity at the origin.

\subsection{The Legendre Transform as a Maximal Envelope}
In previous work of the authors maximal envelopes were used to construct solutions to a Dirichlet problem for the complex Homogeneous Monge-Amp\`ere Equation (HMAE).  The general idea was to start with a concave ``test curve'' $\psi_{\lambda}$ for $\lambda\in (0,c)$ of singular metrics and consider  the Legendre transform 
\[\widehat{\phi}_t : = \sup_{\lambda} \{ \phi_{[\psi_\lambda]} + \lambda t\}^* \quad \text{for }t\in \mathbb R.\]
Letting $w$ be the standard coordinate on the closed unit disc $B\subset \mathbb C$ and changing variables $t=-\ln |w|^2$ we consider $\Phi(z,w):= \widehat{\phi_t}(z)$ as an $S^1$-invariant metric over the product $X\times B$.    In \cite{RossNystrom} it is proved, under some mild assumptions on $\psi_{\lambda}$, that $\widehat{\phi_t}$ is a weak geodesic in the space of positive metrics on $L$ emanating from $\phi$.  That is, $\Phi$ is a positive metric on $\pi^*L$ where $\pi\colon X\times B\to X$ is the projection, that satisfies $MA(\Phi) =0$ over $X\times B$ and $\Phi|_{\partial B} = \phi$.  

 Here we show how the Legendre transform can itself be considered as a maximal envelope over $X\times B$.    Let
 \begin{equation}
\psi' = \sup_{\lambda} \{ \psi_{\lambda}  + \lambda t\}^*.\label{eq:psi'}
\end{equation}

\begin{theorem}
Set $\phi' = \phi + ct$.    Then the Legendre transform of $\psi_{\lambda}$ is given by
\[ \widehat{\phi}_t =  \phi'_{[\psi']}\]
over $X\times B$. 
\end{theorem}

As is well known, an important aspect of the study of Dirichlet problems for the HMAE equation is finding solutions with good regularity properties (see, for example \cite{Guedjbook} for an introduction). We see from what has been said thus far that solutions coming from the Legendre transform construction have as much regularity as the associated maximal envelope.  

\begin{definition}
  We say a test curve is \emph{exponentially H\"older continuous} if the singulairty $\psi'$ is exponentially H\"older continuous on $X\times B$.
\end{definition}

\begin{theorem}\label{thm:reggeo}
  Let $\psi_{\lambda}$ be an exponentially H\"older continuous.  Then for each fixed finite $t\in \mathbb R$ the associated weak geodesic $\widehat{\phi}_t$ is $C^{1,1}$ as a function on $X$, and moreover is locally Lipschitz in the variable $t$.
\end{theorem}

This gives a regularity result for a (reasonably large) class of weak geodesic rays.  We expect this to be suboptimal, and that in fact $\widehat{\phi}_t$ is also $C^{1,1}$ it the variable $t$ (see Remark \ref{rmk:regularity}).  This is very much in the spirit of the regularity result of Phong-Sturm  \cite{Phongregularity} concerning weak geodesics associated to test configurations that will be discussed again below (one observes that, when it applies, the above is neither weaker or stronger than what is proved there).   Certainly if $\psi_{\lambda}$ is the test-curve coming from the degeneration to the normal cone of a divisor in $X$ then it is exponentially H\"older-continuous, and it seems likely that holds for any test curve coming from a test-configuration, but we have not attempted to prove this.

\subsection{Exhaustion functions}  Our final use for maximal envelopes is through the associated exhaustion functions of the equilibrium sets.  Fix a singular metric $\psi\in \PSH(F)$ and consider $H\colon X\to \mathbb R$ given by
\[H= H(\phi,\psi) =  \sup_{\lambda} \{ \phi_{[\lambda\psi]} = \phi \}.\]
It turns out that this exhaustion function is essentially the ``time derivative'' of the associated Legendre function:
\begin{theorem}\label{thm:rightderivative}
  Suppose $\widehat{\phi}_t$ is the Legendre transform associated to the test curve $\psi_{\lambda} = \lambda \psi$ for $\lambda\in (0,1)$.  Then
\[H = \frac{d\widehat{\phi}_t}{dt}\bigg\vert_{t=0^+}. \]
\end{theorem}

A particularly interesting case of the exhaustion function arises when $\psi = \ln |s_D|^2$ where $s_D$ is the defining function of some divisor $D\subset X$.    In this case there is a natural expression for the exhaustion function as a limit of algebraic objects.  Fix $\lambda\in \mathbb Q^+$ and for each large $k$ with $k\lambda\in \mathbb N$ let $\{s_{\alpha}\}$ be an $L^2$ orthonormal basis for $H^0(L^k)$ that is compatible with the filtration  determined by the order of vanishing $\nu_{\alpha}=\ord_D(s_{\alpha})$ along $D$.  That is, for each $j$ the set $\{ s_{\alpha} : \nu_{\alpha}\ge j\}$ is a basis for $H^0(L^k\otimes \mathcal I_D^{j})$. 

\begin{theorem}\label{theorem:gradientmapfromsections}
  We have
\[H = \limsup_{k\to \infty} \frac{\sum_{\alpha} \nu_{\alpha} |s_{\alpha}|_\phi^2}{\sum_{\alpha} k |s_{\alpha}|_\phi^2}\]
almost everywhere on $X$.
\end{theorem}

The previous two theorems can put into context through work of Phong-Sturm on weak geodesics in the space of K\"ahler metrics.    We continue the same notation as above, so  $\{s_{\alpha}\}$ is a basis for $H^0(L^k)$ that respects the filtration by order of vanishing along $D$.  The next is a special case of a construction from \cite{Sturm4}.

\begin{definition}
Let
\[\Phi_k(t):=\frac{1}{k}\ln(\sum_i e^{t\nu_{\alpha}}|s_\alpha|^2)\]
The \emph{Phong-Sturm} ray is the limit
\begin{equation}
\Phi_t:=\lim_{k\to\infty}(\sup_{l\geq k}\Phi_l(t))^*.\label{equationps}
\end{equation}
\end{definition}

In \cite{RossNystrom} it is shown that the Legendre transform of the test curve $\psi_{\lambda} = \lambda |s_D|^2$ equals the Phong-Sturm ray, namely
\[ \widehat{\phi}_t = \Phi(t).\]
So formally differentiating with respect to $t$, and ignoring various exchanges of limits, 
\begin{eqnarray*}
\frac{d}{dt}\bigg\vert_{t=0^+}  \widehat{\phi}_t &=& \frac{d}{dt}\bigg\vert_{t=0^+}  \Phi_t \simeq \frac{d}{dt}\bigg\vert_{t=0^+} \lim_{k\to \infty} \Phi_k(t)\\
&=&\lim_{k\to\infty} \frac{d}{dt}\bigg\vert_{t=0^+} \frac{1}{k}\ln(\sum_i e^{t\nu_{\alpha}}|s_\alpha|_{\phi}^2)  \\
&=& \lim_{k\to \infty} \frac{ \sum \nu_{\alpha} |s_{\alpha}|_{\phi}^2}{k\sum |s_{\alpha}|^2_{\phi}}. 
\end{eqnarray*}

Thus Theorems \ref{thm:rightderivative} and \ref{theorem:gradientmapfromsections} combine to give the same conclusion (almost everywhere), in that both sides are in fact equal to $H$.\medskip

We end with a remark concerning the connection between the exhaustion function $H(\phi, \ln |s_D|^2)$ associated to a divisor $D = \{ s_D=0\}$ and the geometry of the Okounkov body $\Delta(X,L)$ taken with respect to a flag with divisorial part $D$.

\begin{theorem}
  Let $H = H(\phi,\ln |s_D|^2)$ and let $p\colon \Delta(X,L)\to \mathbb R$ be the projection to the first coordinate.  Then
\[H_* (MA(\phi)) = p_* d\sigma\]
where $d\sigma$ denotes the Lebesgue measure on $\mathbb R^n$.
\end{theorem}

This theorem is really nothing more than an unwinding of the definitions and an application of the technical results above.  It partly resembles the Duistermaat-Heckman pushforward property of the moment map in toric geometry, and for this reason we think of $H$ as a kind of weak ``Hamiltonian'' arising from $\phi$ and $D$.\medskip

\noindent{\bf Comparison with other works:} In the time between this article first appearing in preprint form and its publication, there have been some developments directly related that the reader may like to be aware of.  It turns out that the envelopes considered in this paper are intimately related to the class $\mathcal E(X,\omega)$ of $\omega$-plurisubharmonic functions with finite weighted Monge-Amp\'ere energy introduced by Guedj-Zeriahi \cite{GZ}.  In fact, Darvas proves in \cite[Theorem 3]{Darvas} that this class can be characterised using the envelopes $P_{[\psi]}(\phi)$ (see Remark \ref{rmk:darvas}).  Similar envelopes have been studied by Darvas-Rubinstein, with a similar regularity result to the one proved here given in \cite[Theorem 2.5]{DR}. \medskip
 
\noindent {\bf Acknowledgments: } We wish to thank Bo Berndtsson, Robert Berman,  Julien Keller, Reza Seyyedali, Ivan Smith and Richard Thomas for discussions about this work.  We also with to thank Alexander Rashkovskii for pointing out an error in an earlier version of this preprint.  During this project the first author has been supported by a Marie Curie Grant within the 7\textsuperscript{th} European Community Framework Programme and by an EPSRC Career Acceleration Fellowship.

\section{Preliminaries}
\subsection{Singular metrics}

Let $X$ be a K\"ahler manifold of
complex dimension $n,$ and $L$ be a line bundle on $X.$ A
hermitian metric $h=e^{-\phi}$ on $L$ is a choice of
hermitian scalar product on the complex line $L_p$ at each point $p$ on the
manifold. If $f$ is a local holomorphic frame for $L$ on $U_f$, we write \[|f|_h^2=h_f=e^{-\phi_f},\] where $\phi_f$ is a function on $U_f$.  We say that $\phi$ is continuous if this holds for each $\phi_f$ (with analogous definitions for smooth, Lipschitz, $C^{1,1}$ etc.).   It is standard abuse of notation to let $\phi$ denote the metric $h=e^{-\phi}$ and to confuse $\phi$ with $\phi_f$ if a given frame is to be understood.  Thus if $\phi$ is a metric on $L,$ $k\phi$ is a metric on $kL:=L^{\otimes k}.$

The curvature of a smooth metric is given by
$dd^c\phi$ which is the $(1,1)$-form locally defined as $dd^c\phi_f,$
where $f$ is any local holomorphic frame and $d^c$ is the differential operator \[\frac{i}{2\pi}(\partial-\bar{\partial}),\]
so $dd^c=(i/\pi) \partial \bar{\partial}.$   The
curvature form of a smooth metric $\phi$ is a
representative for the first Chern class of $L,$ denoted by $c_1(L).$
A smooth metric $\phi$ is said to be \emph{strictly positive} if  $dd^c\phi$
is strictly positive as a $(1,1)$-form, i.e.\ if for any local holomorphic frame
$f,$ the function $\phi_f$ is strictly plurisubharmonic.

A \emph{positive singular metric} is a metric that can be written as
$\psi:=\phi+u,$ where $\phi$ is a smooth metric and $u$ is a
$dd^c\phi$-psh function, i.e.\ $u$ is upper semicontinuous and
$dd^c\psi:=dd^c \phi+dd^c u$ is a positive $(1,1)$-current. For
convenience we also allow $u\equiv -\infty.$  The singular locus of $\psi$ will be denoted by $\Sing(\psi)$ is the set on which $\psi$ is not locally bounded.  We let $\PSH(L)$ denote the space of positive singular metrics on $L$.   If $\Sing(\psi)$ is empty we say $\psi$ is \emph{locally bounded} (we will mostly consider the case $X$ is compact in which case this is equivalent to being \emph{globally bounded}).

We note that $\PSH(L)$ is a convex set, since any convex combination
of positive metrics yields a positive metric. Moreover if $\psi_i \in \PSH(L)$ for $i\in I$ are uniformly bounded above by some fixed positive metric, then the upper semicontinuous regularisation of the supremum denoted by $(\sup\{\psi_i : i\in I\})^*$ lies in $\PSH(L)$ as well.  If $\psi\in \PSH(L),$ then the translate $\psi+c$ where $c$ is a real constant
is also in $\PSH(L).$

A plurisubharmonic function $u$ on a set $W$ is \emph{maximal} if for every relative compact $U\subset W$ and upper semicontinuous function $v$ on $\overline{U}$ with $v\in \PSH(U)$ the inequality $v\le u$ on $\partial U$ implies $v\le u$ on all of $U$.  

If $\psi$ and $\phi$ are metrics on $L$ and there exists a constant $C$ such that $\psi \leq \phi+C,$ we say that $\psi$ is \emph{more singular} than $\phi$.   When specific mention of the constant $C$ is unimportant we shall write this as $\psi\le \phi + O(1)$.  Of course $\psi\le \phi +O(1)$ if and only if $\psi\le \phi+O(1)$ holds on some neighbourhood of $\Sing(\phi)$ and we will use this in the sequel without further comment.   

More generally, if $\psi$ in a metric on $L_1$ and $\phi$ a metric on $L_2$ we will write $\psi\le \phi + O(1)$ to mean there is a locally bounded metric $\tau$ on $L_1\otimes L_2^*$ such that $\psi\le \phi + \tau$.  The condition $\psi \le \phi +O(1)$ and $\phi \le \psi + O(1)$ is an equivalence relation, which we denote by $\psi\sim \phi + O(1)$, and following \cite{Guedj} an equivalence class $[\psi]$ is called a \emph{singularity type}.

If $\psi_i$ is a metric on $F_i$ for $i=1,2$ then by abuse of notation we will occasionally write
\begin{equation}
  \label{eq:abusesupremum}
\sup\{\psi_1,\psi_2\}  
\end{equation}
to mean the metric on $F_1+F_2$ given by $\sup\{\psi_1 + \phi_{F_2}, \phi_{F_1} + \psi_2\}$ where $\phi_{F_i}$ is a choice of globally bounded metric on $F_i$.  Thus the singularity type of $\sup\{\psi_1,\psi_2\}$ is independent of choice of $\phi_{F_i}$.  

Given a coherent analytic ideal sheaf $\mathcal I\subset \mathcal O_X$ and a constant $c>0$ we say that $\psi$ has \emph{analytic singularities modeled on} $(\mathcal I,c)$ if $X$ is covered by open sets $U$ on which we can write
\begin{equation}
  \label{eq:analyticsingularities}
\psi=c(\ln \sum |f_i|^{2}) + u  
\end{equation}
where $f_i$ are generators for $\mathcal I(U)$, and $u$ is a smooth function.  If $\mathcal I$ is algebraic, $c$ is rational and we can arrange this to hold in the Zariski topology then we say $\psi$ has \emph{algebraic singularities} modeled on $(\mathcal I,c)$. 

We say that a singular metric $\psi$ is \emph{exponentially H\"older continuous} with exponent $c>0$, if it is smooth away from $\Sing(\psi)$ and over the singular locus satisfies
\[ |e^{\psi(x)} - e^{\psi(y)} | \le C|x-y|^c\]
for some constant $C$ (here we are taking a local expression for $\psi$ thought of as a function on some coordinate chart $U$ and the norm on the right hand side is taken to be the Euclidean norm on the coordinates).  A metric is \emph{exponentially Lipschitz} if it is H\"older continuous  with exponent $c=1$.   Note that if $\psi$ has analytic singularities as in \eqref{eq:analyticsingularities} then it is exponentially H\"older continuous with exponent $c$.

Given a metric $\psi$ the \emph{multiplier ideal} $\mathcal I(\psi)$ is the ideal generated locally by holomorphic functions $f$ such that $|f|^2e^{-\psi}\in L^1_{loc}$.   So if $\psi\le \phi + O(1)$ then  clearly $\mathcal I(\psi)\subset \mathcal I(\phi)$.

The \emph{Monge-Amp\`ere measure} of a metric $\phi$ is defined as the
positive measure \[MA(\phi):=(1/n!)(dd^c \phi)^n.\]  When $\phi$ is smooth this is defined by taking the wedge product of the $(1,1)$ forms $dd^c \phi$ in the usual sense.   Through the fundamental work of Bedford-Taylor, the Monge-Amp\`ere measure can in fact be defined on the set on which $\phi$ is locally bounded, and this measure does not put any mass on pluripolar sets (i.e.\ sets that are locally contained in the unbounded locus of a local plurisubharmonic function).

\subsection{Augmented Base Locus}

Let $L$ be a big line bundle.  The base locus of $L$ is the set \[Bs(L) = \bigcap_{s\in H^0(L)} \{ x: s(x)= 0\}\] and the stable base locus is $\mathbb B(L) = \bigcap_k Bs(kL)$.  We denote by $\mathbb B_+(L)$ the augmented base locus of $L$ which is given by
\[ \mathbb B_+(L) = \mathbb B(L-\epsilon A) \quad \text{ for any small rational }\epsilon>0\]
where $A$ is any fixed ample line bundle on $X$.  It is a fact that $L$ is ample if and only if $\mathbb B_+(L)$ is empty, and is big if and only if $\mathbb B_+(L)\neq X$ \cite[Example 1.7]{Ein}.

\subsection{Partial Bergman Functions}
Now suppose we fix a smooth volume form $dV$ on $X$.  Then any metric $\phi$ on $L$ induces an $L^2$-inner product on $H^0(L^k)$ for all $k$,  whose norm is given by
\begin{equation}
  \label{eq:l2norm}
  \|s\|_{\phi,dV}^2 = \int_X |s|_{k\phi}^2 dV \quad \text{for }s\in H^0(L^k).
\end{equation}
We will omit the $dV$ from the notation when the volume form is understood.  

\begin{definition}\label{thm:equilibriummeasure}
  The \emph{partial Bergman function} associated to $\phi$ and $\psi$ is the function
\[B_k(\phi,\psi) = \sum_{\alpha} |s_{\alpha}|_{\phi}^2\]
where $\{s_{\alpha}\}$ is any $L^2$-orthonormal basis for $H^0(\mathcal I(k\psi) L^k)$.
\end{definition}

If $\psi$ is locally bounded then the associated multiplier ideal sheaf is trivial, and  $B_k(\phi,\psi)$ becomes the usual Bergman function for $\phi$ which for simplicity we shall denote by $B_k(\phi)$.   Thus $k^{-1}\ln B_k(\phi,\psi)+\phi$ is a metric on $L$ with singularities modeled on $(\mathcal I(k\psi), k^{-1})$.

\begin{remark}
  The partial Bergman function depends on the choice of smooth volume form, but it is easy to verify directly that the limit $k^{-1} \ln B(\phi,\psi)$ as $k$ tends to infinity does not since the quotient of any two volume forms is globally bounded.
\end{remark}

\begin{example}\label{example:submanifold}
  Let $Y\subset X$ be a smooth subvariety of codimension $r$ which is given by the intersection of a finite number of sections of some line bundle $F$.  Then we can define a singular metric $\psi = \ln \sum_i |s_i|^2$.  To calculate the multiplier ideal let $\pi\colon \tilde{X}\to X$ be the blowup along $Y$ with exceptional divisor $E$ and canonical divisor $K_{\tilde{X}} = \pi^* K_X + (r-1)E$.    Then, by smoothness of $Y$, $\pi_* \mathcal O_{\tilde{X}}(-uE) = \mathcal I_Y^u$ for all $u\ge 0$.   Following \cite[5.9]{Demaillybook} one computes $\mathcal I(k\psi) = \pi_* \mathcal O_{\tilde{X}}((r-1-k)E) = \mathcal I_Y^{k-r+1}$.    Thus $B_k(\phi,\psi)$ is precisely the partial Bergman kernel consisting of sections that vanish to a particular order along $Y$ (which for toric manifolds is studied in \cite{Pokorny, PokornySinger}).
\end{example}

\subsection{The Ohsawa-Takegoshi extension theorem}\label{sec:ohsawa}

We will need the following version of the Ohsawa-Takegoshi extension theorem.  To state it cleanly we shall say that a metric $\phi_F$ on a vector bundle $F$ has the \emph{extension property with constant $C$} if for any $x\in X$ and $\zeta\in F_x$ there is an $s\in H^0(X,F)$ with $s(x) = \zeta$ and 
\[ \|s\|_{\phi_F} \le C |s(x)|_{\phi_F}.\]

\begin{theorem}(Ohsawa-Takegoshi)\label{thm:ohsawa}
Suppose $c>0$ is given.  Then there exists a $k'$ and a $C'$ such that the following holds:  for all $k\ge k'$ and all $L'$ and $F'$ with singular metrics $\phi_{L'}$ and $\phi_{F'}$ on $L',F'$ respectively such that
\begin{equation}
dd^c \phi_{L'} \ge c\omega \quad dd^c \phi_{F'} \ge -c\omega\label{eq:curvature}
\end{equation}
the metric $k\phi_{L'} + \phi_{F'}$ has the extension property with constant $C'$.
\end{theorem}

This well-known statement is a consequence of the more general result proved in \cite[Proposition 12.4]{Demailly}

\section{Maximal Envelopes}\label{sec:envelopes}
Fix a complex manifold $X$ (which we shall assume is compact unless indicated otherwise) along with line bundles $L$ and $F$.  Let $\phi$ be a (not necessarily positive) continuous metric on $L$ and pick $\psi\in \PSH(F)$.   Fix also a smooth metric $\phi_F$ on $F$ and define
\[\psi' = \phi - \phi_F + \psi.\]
\begin{definition}
Let $P_{\psi'}$ be the envelope
  \begin{equation}
P_{\psi'}\phi:=  \sup\{\gamma \leq \min\{\phi,\psi'\}, \gamma\in \PSH(L)\},\label{eq:envelope}
\end{equation}
and define  \[P_{[\psi]}\phi:=\lim_{C\to
    \infty}P_{\psi'+C}\phi=\sup\{\gamma \leq \phi, \gamma \le \psi' + O(1),
  \gamma\in \PSH(L)\}.\]
\end{definition}

The notation is justified by the observation $P_{[\psi]}\phi$ is independent of the choice of $\phi_F$ (because the latter  is globally bounded) and thus depends only on the singularity type $[\psi]$.    Since $\min\{\phi,\psi'\}$ is upper semicontinuous, it follows
that the upper semicontinuous regularisation of $P_{\psi'}\phi$ is
still less than $\min\{\phi,\psi'\},$ and thus $P_{\psi'}\phi\in
\PSH(L).$   Hence
$P_{\psi'}(P_{\psi'}\phi)=P_{\psi'}\phi,$ i.e.\ $P_{\psi}$ a projection operator to
$\PSH(L).$   Clearly $P_{\psi}\phi$ is monotone with respect to both $\psi$ and
$\phi$.   We shall always assume that $L$ is pseudoeffective, otherwise $\PSH(L)$ will be empty and the above envelopes
will be identically $-\infty$.

\begin{definition}\label{definitionmaximal}
  The \emph{maximal envelope} of $\phi$ with respect to the
  singularity type $[\psi]$ is 
\[\phi_{[\psi]} := (P_{[\psi]}\phi)^*\]
where the star denotes the upper-semicontinuous regularisation.      The \emph{equilibrium set} associated to $\phi$ and $\psi$ is
\[ D= D(\phi,\psi) = \{ x \in X : \phi_{[\psi]}(x) = \phi(x)\}\]
\end{definition}

Clearly then $\phi_{[\psi]}\in \PSH(L)$ and $\phi_{[\psi]}\le \phi$.

\begin{example}[Trivial Singularities]\label{example:locallybounded}
If $\psi$ is locally bounded then $\phi\le \psi'+C$ for $C$ sufficiently large, and thus
\[\phi_{[\psi]} = P_{\psi'+C} \phi = \sup\{ \gamma\le \phi : \gamma\in \PSH(L)\}.\]
These are exactly the envelopes considered by Berman in \cite{Berman}.
\end{example}

\begin{remark}
  In the locally bounded case, maximal envelopes are examples of metrics with minimal singularities in that if $\gamma$ is any other positive metric on $L$ then $\gamma \le \phi_{[\psi]} + O(1)$.    
\end{remark}

\begin{example}[Divisorial singularities]
  Suppose $D$ is a smooth divisor in $X$ and $F= \mathcal O_X(D)$ with singular metric $\psi = \ln |s_D|^2$ where $s_D$ is the defining function for $D$.  Then
\[\phi_{[\psi]}=\sup \{\gamma \le \phi : \gamma\in PSH(L), \nu_D(\gamma)\ge 1\}^*\]where $\nu_D$ denotes the Lelong number along $D$.  This case is considered by Berman \cite[Sec.\ 4]{Berman1}.
\end{example}

\begin{example}[Pluricomplex Green Function]
For non-compact $X$, taking $L$ to be the trivial bundle and $\phi=0$ the trivial metric, the maximal envelope becomes the pluricomplex Green function on $X$.  When $\psi$ has analytic singularities this has been studied by Rashkovskii-Sigurdsson \cite{Rashkovskii}.  We remark in passing that the pluricomplex Green function on compact manifolds with boundary has recently been considered by Phong-Sturm \cite{Phong5}, but is more commonly studied on domains in $\mathbb C^n$ along with a boundary condition, for which it has a long and rich history (see \cite{Bracci} and the references therein).
\end{example}

\begin{remark}[Invariance under holomorphic automorphisms]
  If $\phi$ and $\psi$ are invariant under some group $G$ of holomorphic automorphism of $(X,L)$ then the same is true of $P_{\psi'}\phi, P_{[\psi]}$ and $\phi_{[\psi]}$.  The proof is immediate, for if $\theta$ is such an automorphism then $P_{\psi'}\phi\circ \theta\in PSH(L)$ and is bounded by $\min\{\phi,\psi'\}$ and thus also bounded by $P_{\psi'}\phi$.  Applying to the inverse of $\theta$ then yields $P_{\psi'} \phi \circ \theta = P_{\psi'}\phi$.  Thus there is no loss in replacing the envelope in \eqref{eq:envelope} with those $\gamma$ that are invariant under $G$.
\end{remark}

\begin{example}[Toric metrics]\label{example:toric}
Consider now the case of a toric variety $X$ with torus invariant $L$ which we assume is ample.  Let $\Delta$ be the associated Delzant polytope in $\mathbb R^n$.  Letting $z_i$ be complex coordinates on the large torus in $X$, any hermitian metric $\phi$ on $L$ descends to a convex function on $\mathbb R^n$ after the change of variables $x_i = \ln |z_i|^2$ which by abuse of notation we denote by $\phi(x)$.  Moreover $\phi$ is positive if and only if $\phi(x)$ is convex.  Thus if $\psi$ is locally bounded (and so irrelevant) then $\phi_{[\psi]}(x)$ is simply the convex hull the graph of $\phi(x)$  (see \cite[5.2]{Berman}).   

Suppose instead we have $\psi = \lambda \ln |z_1|^2$ for some fixed $\lambda>0$.  Then $\phi_{[\psi]} (x)$ is the supremum of all convex functions $\gamma$ on $\mathbb R^n$  bounded by $\phi$ such that for $x_1\gg 0$ we have $\gamma\le \lambda x_1+C$ for some  $C$.  
\end{example}

\begin{remark}\label{rmk:darvas}
In the first version of this paper it was noted that it is not obvious if the maximal envelope $\phi_{[\psi]}$ has the same singularity type as $\psi$.  When $\psi$ has analytic singularities this can be shown rather easily by passing to a smooth resolution (see \cite{Rashkovskii} which also contains an alternative proof).  This topic has since been taken up by Darvas \cite[Theorem 3]{Darvas} who shows that in general this is not the case, and gives an interesting criterion for it to hold in terms of a certain natural class $\mathcal E(X,L)$ of positive metrics on $L$. 
\end{remark}

\section{Extension of results of Berman}

In this section we extend some results of Berman to the maximal envelopes considered in this paper.  What follows is essentially due to Berman, which in turn is based on the work of Bedford-Taylor \cite{Bedford3}.  The exposition here follows closely \cite{Berman} which in fact announces that such an extension should hold \cite[Sec 1.3]{Berman}.  The related work \cite[Sec 4]{Berman1} deals with the case of envelopes that appear from order of vanishing along a divisor and \cite{Hisamoto} proves related results in the case of general graded linear series.

\subsection{Logarithmically homogeneous plurisubharmonic functions}

We first describe a general framework which allows us to pass from  metrics over a compact space to metrics over an auxiliary non-compact space.    Recalling that $L$ is a line bundle over $X$, let $Y$ be the total space of the dual bundle $L^*$ and  $\pi\colon Y\to X$ be the projection.

Consider $Y$ as a subset of the compactification $\overline{Y}:= \mathbb P(L^*\oplus \underline{\mathbb C})$ where $\underline{\mathbb C}$ denotes the trivial line bundle over $X$.  Then over $\overline{Y}$ the hyperplane line bundle $\mathcal O_{\overline{Y}}(1)$ has a section $s\in H^0(\overline{Y}, \mathcal O_{\overline{Y}}(1))$ given by the constant section on the factor $\underline{\mathbb C}$.    We let
$$ \zeta : = \ln |s|^2 \in PSH(\mathcal O_{\overline{Y}}(1))$$
which is well-defined up to the addition of a constant.   We write $\mathcal O_Y(1)$ for the restriction of $\mathcal O_{\overline{Y}}(1)$ to $Y$, and denote the restriction of $\zeta$ to $Y$ by the same letter.
Thus, concretely, if $w$ is a local coordinate on the fibre direction of $Y$, then  $s$ is given locally by the equation $w=0$ and so
$$ \zeta = \ln |w|^2.$$
Finally, to any metric $\gamma$ on $L$ we define
\[\widehat{\gamma} = \pi^* \gamma + \zeta. \]

\begin{definition}
The set of \emph{logarithmically homogeneous plurisubharmonic functions} on $Y$ is defined to be
  \begin{equation}
\PSH_h(Y) := \{ \chi \in PSH(\pi^*L\otimes \mathcal O_Y(1)) : \chi(\lambda y ) = \ln |\lambda|^2 + \chi(y) : \lambda \in \mathbb C^*\},\label{eq:pshh}
\end{equation}
where the multiplication by $\mathbb C^*$ is taken in the fibre direction of $Y$.

\end{definition}
Thus the map $\gamma\mapsto \widehat{\gamma}$ gives a bijection between $PSH(L)$ and $\PSH_h(Y)$.  Moreover this bijection respects taking envelopes, which we make precise in the following lemma.  Define $P_{\widehat{\psi'}}\widehat{\phi}$ and $P_{[\widehat{\psi}]}\widehat{\phi}$ and $\widehat{\phi}_{[\widehat{\psi}]}$ exactly as in Section \ref{sec:envelopes}, where now the supremum in \eqref{eq:envelope} is taken over all $\gamma$ in $PSH_h(Y)$.

\begin{lemma}\label{lemma:hatrespectenvelopes}
We have $\widehat{P_{\psi'}\phi} = P_{\widehat{\psi'}}\widehat{\phi}$ and 
$\widehat{P_{[\psi]}\phi} = P_{[\widehat{\psi}]}\widehat{\phi}$ and
$\widehat{\phi_{[\psi]}} = \widehat{\phi}_{[\widehat{\psi}]}$.
\end{lemma}

\begin{proof}
For $\phi,\psi'\in PSH(L)$
  \begin{eqnarray*}
    \widehat{P_{\psi'}\phi} &=& \pi^* P_{\psi'}\phi+ \zeta=\sup_{\gamma\in\PSH(L)}\{\pi^*\gamma + \zeta : \gamma\le \min\{\phi,\psi'\}\}\\
&=&\sup_{\widehat{\gamma}\in \PSH_h(Y)}\{\widehat{\gamma} : \widehat{\gamma} \le \min\{\widehat{\phi},\widehat{\psi'}\}\}
  \end{eqnarray*}
which gives the first identity.  Moreover 
\begin{eqnarray*}
  \widehat{P_{[\psi]}\phi} &=& \zeta + \pi^* \lim_{C\to \infty} P_{\psi'+C} \phi = \lim_{C\to \infty}( \pi^* P_{\psi'+C} \phi+\zeta)\\
   &=& \lim_{C\to \infty} \widehat{P_{\psi'+C}\phi}= \lim_{C\to \infty} P_{\widehat{\psi'}+C}\widehat{\phi} = P_{[\widehat{\psi}]}\widehat{\phi}
\end{eqnarray*}
where the penultimate equality uses part (a) and $\widehat{\psi+C} = \widehat{\psi} +C$.  Finally
\begin{eqnarray*}
  (P_{[\widehat{\psi}]} \widehat{\phi})^* &=& (\widehat{P_{[\psi]}\phi})^*= (\pi^*(P_{[\psi]}\phi) +\zeta) ^*\\
&=& \pi^*(P_{[\psi]}\phi^*)+\zeta= \phi_{[\psi]}= \widehat{\phi_{[\psi]}}+\zeta
\end{eqnarray*}
where the third inequality uses the fact that $\zeta$ is upper-semicontinuous and the local (and elementary) fact that $(f(w) + g(z))^* = f^*(w) + g^*(z)$.
\end{proof}

\subsection{Exponential holomorphic coordinates}

Now fix a continuous metric $\phi$ on $L$.    We choose a smooth metric $\phi_{F}$ on $F$ and set $\psi':=\phi-\phi_{F}+\psi$.   To ease notation let
\[ \tau_C := P_{\psi'+C} \phi\]
so by definition
\[ \phi_{[\psi]} = (\lim_{C\to \infty} \tau_C)^*.\]
Our aim is to show regularity around a fixed point
\[x_0 \in X-\mathbb B_+(L-F)\cup \Sing(\psi).\]
From the assumption that $L-F$ is big there is, for $k\gg 0$, a Kodaira decomposition $k(L-F) = A+E$ where $A$ is ample and $E$ is an effective divisor in $X$.  Since $\mathbb B_+(L-F)$ is the intersection over all such $E$ as $k$ varies, we can arrange so that $x_0\notin E$.    Furthermore there exists a positive metric $\phi_+$ on $L-F$ of the form
\[\phi_+ = k^{-1}(\phi_{A} + \ln |s_{E}|^2),\]
where $\phi_{A}$ is a smooth positive metric on $A$ and $s_{E}$ is the defining section of $E$.   

Observe that for any $k\ge 1$.
\begin{eqnarray}\label{rescaleenvelopes}
  P_{k\psi'} (k\phi) = k P_{\psi'}\phi \quad\text{and}\quad (k\phi)_{[k\psi]} = k \phi_{[\psi]}.
\end{eqnarray}
Thus by scaling $L$, $\phi$ and $\psi$ we may assume without loss of generality that $k=1$.  Moreover since $\psi$ was assumed to be exponentially H\"older continuous, $k\psi$ is exponentially Lipschitz for sufficiently large $k$.  Thus there is no loss in assuming $\psi$ is exponentially Lipschitz.  Furthermore, by subtracting a constant from $\phi_+$ if necessary we can also arrange
\begin{equation}
  \phi_+ + \psi \le \tau_C\text{ for all } C\ge 0.\label{eq:boundonphi+}
\end{equation}

Now let
\[ Y_0 = Y - (j(X) \cup \pi^{-1}( E \cup \Sing(\psi))\]
where $j\colon X\to Y$ is the inclusion of $X$ as the zero section, and define a disc bundle in $Y$ by
\[U=\{\widehat{\phi_{+}}+\pi^*\psi \le  1\}.\]
To show regularity of $\phi_{[\psi]}$ near $x_0$ it is, by Lemma \ref{lemma:hatrespectenvelopes}, sufficient to show regularity of $\widehat{\phi}_{[\widehat{\psi}]}$ near a chosen point in the fibre $\pi^{-1}(x_0)$.  To this end pick such a point, 
\[y_0 \in \pi^{-1}(x_0)\cap U - j(X),\]
so we have $y_0\in U\cap Y_0$.   The method of proof now follows the original approach of Bedford-Taylor, using particular holomorphic coordinates around $y_0$ that arise from holomorphic vector fields constructed in the next lemma (in fact this is the reason for passing to the auxiliary space $Y$ since, in general, $X$ may have no holomorphic vector fields at all).

\begin{lemma}\label{lem:existencevectorfields}\
There exist global holomorphic vector fields $V_1,\ldots,V_{n+1}$ on $Y$ whose restriction at $y_0$ span  $T_{y_0}Y$.   Moreover, given any positive integer $m$ there is a constant $C_m$ so that these vector fields can be made to satisfy
  \begin{equation}
 |V_i(z,w)|^2 \le C_m\min\{ |w|^{2m}, |s_{E}(z)|^{2m} e^{m\psi(z)}\}\text{ on }U.\label{eq:estimatevectorfield}
\end{equation}

\end{lemma}

\begin{proof}
The space $Y$ has the compactification $\overline{Y} = \mathbb P(L^*\oplus\underline{\mathbb C})$, and we shall denote the tautological bundle by $\mathcal O_{\overline{Y}}(-1)$.   Fix a continuous  metric $\phi_{E}$ on $E$ and consider the metric  on $L$ given by
\[\phi_{+,k} = \phi_{A} + (1 + k^{-1/2})( \ln |s_{E}|^2+\psi ) - k^{-1/2}\sigma\]where $\sigma$ is a fixed choice of smooth metric on $E + F$. Observe that for $k\gg 0$ we will have $dd^c \phi_{+,k}\ge(1/2)( dd^c \phi_+)$ since $\psi$ is positive.  

Now define a metric on  the line bundle $\overline{L} = \pi^* L^{k} \otimes \mathcal O_{\overline{Y}}(1)$ over $Y$ by
\[\overline{\phi}: = \pi^* k \phi_{+,k} + \ln ( 1 + e^{\widehat{\phi}})\]
which extends to a metric over all of $\overline{Y}$ which has strictly positive curvature for $k$ sufficiently large.  Then by the Ohsawa-Takegoshi Theorem \eqref{thm:ohsawa} the vector bundle $T\overline{Y}\otimes \overline{L}^{l}$ is globally generated for $l\gg 0$.  Fix such an $l$ with $kl\ge m$.   Thus there are sections $W_1,\ldots,W_{n+1}$ whose evaluation at $y_0$ span $TY_{y_0}\overline{L}^{l}$.   Thinking of $w$ as the tautological section of $\pi^*L^*$ we let $V_i$ be the restriction of $W_i w^{kl}$ to $Y$.  Thus the $V_i$ are vector fields on $Y$ and whose evaluation at $y_0$ span $T_{y_0}Y$.

Now the vector fields $W_i$ are holomorphic, so have bounded supremum norm over $\bar{Y}$ and so also over $Y$.  Thus locally on a fixed neighbourhood in $Y$ we have
\begin{eqnarray*}
|V_i(z,w)|^2 &\le& C_1 |w|^{2kl} e^{l\overline{\phi}}= C_1|w|^{kl} e^{kl\phi_{+,k}} e^{l\hat{\phi}} \\
&\le&C_2|w|^{2kl}((|s_{E}|^2e^{\psi})^{k(1+k^{-1/2})}|w|^2)^l\\
&\le& C_3 (|w|^2 |s_{E}|^2 e^{\psi})^{(k+1)l} |s_{E}|^{2m} e^{m\psi}
\end{eqnarray*}
which yields the required bound on $|V_i|$ since $\widehat{\phi_+} +\pi^*\psi = \ln (|w|^2|s_{E}|^2e^{\psi}) + \phi_{A}$, so $\|w\|^2\|s_{E}|^2e^{\psi}$ is bounded on $U= \{\widehat{\phi_+} +\pi^*\psi\le 1\}$. 
\end{proof}

\begin{remark}
It is the hypothesis that $L-F$ is big that allows for the fact that the conclusion of the previous Lemma is stronger than \cite[Lemma 3.6]{Berman} on which it is based.
\end{remark}

\begin{lemma}\label{lemma:flowexists}
  Let $V$ be a smooth vector field on a manifold $Y$, and $Y_0\subset Y$ be such that $V=0$ on $Y-Y_0$.  Suppose $U\subset Y$ is a closed subset, and  $\alpha\colon Y_0\to Y_0$ is a diffeomorphism.

$W:=\alpha(U\cap Y_0)$ is relatively compact and so that the vector field $\alpha_* V$ can be extended to a vector field on $\overline{W}$ that vanishes on the boundary $\overline{W}-W$.    Then there exists a $t_0$ such that the flow $\exp(tV)(y)$ exists for all $|t|\le t_0$ and all $y\in U$.
\end{lemma}
\begin{proof}
Denote the extension of $\alpha_*V$ by $\tilde{V}$.  By compactness of $\overline{W}$ there is a $t_0$ such that flow of $\tilde{V}$ exists for any $|t|\le t_0$ and initial point $y\in \overline{W}$.  Since $\tilde{V}$ vanishes on $\overline{W}-W$, we have that if $y\in W$ then this flow remains completely within $W$.  However this is precisely the image (under $\alpha$) of the flow of $V$, which proves that the flow of $V$ exists for all $|t|\le t_0$ and $y\in U\cap Y_0$.  Finally if $y\in U- Y_0$ then $y$ is fixed by the flow of $V$, and thus the flow exists for all time in this case as well.
\end{proof}

\begin{lemma}\label{lem:flowexists}
There exists a $t_0>0$ such that the flow $y\mapsto \exp((\sum_i \lambda_i V_i)y)$ exists for all $(\lambda,y)$ such that $|\lambda|\le t_0$ and $y\in U$. 
\end{lemma}
\begin{proof}
We apply the previous lemma to the vector field $V=V_i$ for $i=1,\ldots, n+1$ and $Y_0 = Y-j(X)\cup \pi^{-1}(E\cup \Sing(\psi))$.  Picking smooth metrics $\phi_F$ and $\phi_E$ on $F$ and $E$ respectively, define $\alpha\colon Y_0 \to Y_0$ by
\[\alpha(\zeta) = f \zeta   \quad \text{ where } f = |s_{E}|_{\phi_{E}} e^{(\psi-\phi_F)/2}\]
and the multiplication is to be understood in the fibre direction of $Y$, and  $\psi-\phi_F$ is thought of as a function on $X$.  Recall also that included in the assumption that $\psi$ is exponentially Lipschitz  is that $e^\psi$ is smooth away from $\Sing(\psi)$.    Thus, by construction, $\alpha$ is a diffeomorphism.  Now a simple calculation reveals
\[ (\widehat{\phi}_+ + \pi^* \psi)\circ \alpha = \widehat{\phi_A} + \pi^*(\phi_F - \phi_E) \quad \text{ on } Y_0,\]
and thus if $W = \alpha(U\cap Y_0)$ then $\overline{W} = \overline{\alpha(U\cap Y_0)} = \{ \widehat{\phi_A} + \pi^* (\phi_F-\phi_E)\le 1\}$ which is compact. Now set $\tilde{V} = \alpha_* V$ and suppose that locally in $(z,w)$ coordinate $V= v_z \frac{\partial}{\partial dz} + v_w \frac{\partial}{\partial dw}$ and similarly for $\tilde{V}$.   Writing $\alpha$ locally as $(z,w)\mapsto (\tilde{z},\tilde{w})$ we have $\tilde{v}_z(\tilde{z},\tilde{w}) = v_z(z,w)$ and 
\[\tilde{v}_w (\tilde{z}, \tilde{w})= \frac{\partial f}{\partial z} \frac{v_z}{f} \tilde{w} + v_w(z,w) f(z).\]
Now the fact that $\psi$ is exponentially Lipschitz implies $\frac{\partial f}{\partial z}$ is globally bounded.  Thus taking $m=2$ in the estimate in \eqref{eq:estimatevectorfield} bounds $f^{-1} v_z$ by $|s_{E}|e^\psi$.  Hence $\tilde{V}$ extends over $\overline{W}$ and vanishes on $\overline{W}-W$, so the previous lemma applies to give the result.
\end{proof}

\subsection{Proof of Lipschitz regularity}

Suppose $V_1,\ldots,V_{n+1}$ are the vector fields provided by Lemma \ref{lem:existencevectorfields} whose evaluation at $y_0$ span $T_{y_0}Y$ taking $m=2$.   Consider the flow 
\[\theta_{\lambda}(y) = \exp(\sum_{i=1}^{n+1} \lambda_i V_i)(y),\]
which by Lemma \ref{lem:flowexists} is well defined for $y\in U$ and $|\lambda|$ sufficiently small.  Then for any function $f$ on $U$ denote the pullback function by
\[f^{\lambda}(y) := f(\theta_{\lambda}(y))\]  
Thus to show that a function $f$ on $Y$ is Lipschitz near $y_0$ it is sufficient to prove
\[ |f^{\lambda}(y) - f(y)|\le C|\lambda|\]
for some constant $C$ and $y$ in some neighbourhood of $y_0$.

\begin{lemma}\label{lem:boundlambdafunction}\
  \begin{enumerate}
\item Let $\phi_0$ be a metric on $L$ such that $\phi_0-\delta_1 |s_{E}|^2-\delta_2 \psi$ is smooth for some constant $\delta_i\ge 0$.  Then  there is a constant $C_{\alpha}$ so that 
\[| \widehat{\phi_0}^{\lambda} - \widehat{\phi_0} | \le C_{\alpha} |\lambda|\]
on $U\cap Y_0$.
\item Let $\phi_0$ be a metric on $L$ such that $\phi_0-\delta |s_{E}|^2$ is smooth for some constant $\delta\ge 0$.  Then  there is a constant $C_{\alpha}$ so that for all multiindices $\alpha$ of total order at most two
\[| \partial^{\alpha}_{z,w} ( \widehat{\phi_0}^{\lambda} - \widehat{\phi_0}) | \le C_{\alpha} |\lambda|\]
on $U\cap Y_0$.
\end{enumerate}
\end{lemma}
\begin{proof}
For $y\in U\cap Y_0$ let
\[\gamma(t) = \exp( t \sum_{i=1}^{n+1} \lambda_i V_i)(y) ,\]
so $\gamma(0) = y$ and $\gamma(1) = \theta_{\lambda}(y)$.  Now $\widehat{\phi_0}$ is smooth away from $j(X)\cup \pi^{-1}(E\cup \Sing(\psi))$, so we can write
\begin{equation}
 \widehat{\phi_0}^{\lambda}(y)- \widehat{\phi_0}(y) = \int_0^1 d\widehat{\phi_0}\vert_{\gamma(t)}\left(\frac{d\gamma}{dt}\right) dt.\label{eq:ftc}
\end{equation}

We shall prove the first statement.    Pick local coordinates $z$ on $X$ and $w$ on $L$ around some point in $Y$.    Then
\begin{equation}\label{eq:dphi0}
d\widehat{\phi_0}( \frac{d\gamma}{dt}) = \left( w^{-1} dw + d \phi_0\right) (\sum \lambda_i V_i)
\end{equation} 
From \eqref{eq:estimatevectorfield}, $|V_i|\le C_1|w|^2$, so $|w^{-1} dw(V_i)|$ is uniformly bounded and thus \[|w^{-1} dw(\sum_i \lambda_i V_i)| \le C_2|\lambda|,\] for some constant $C_2$.     

If our chart lies outside $\pi^{-1}(E\cup \Sing(\psi))$ then $\widehat{\phi_0}$ is smooth so the second term in \eqref{eq:dphi0} is clearly bounded by a constant times $|\lambda|$.     To deal with the case that the chart meets $E\cup \Sing(\psi)$, we use the hypothesis $\phi_0 - \delta_1 \ln |s_{E}|^2-\delta_2\psi$ is smooth to deduce
\[ |d \widehat{\phi_0}(V_i)| \le C_4   + \delta_1 \left\vert \frac{1}{s_{E}} \frac{\partial s_{E}}{\partial x}(V_i)\right\vert + \delta_2 e^{-\psi}\left\vert \frac{\partial e^\psi}{\partial x}(V_i) \right\vert\]
where the derivative of $e^\psi$ exists weakly (and is globally bounded) by the assumption that $e^{\psi}$ is Lipschitz. Now using \eqref{eq:estimatevectorfield}, $|V_i|^2\le C_5 \min\{ |s_{E}|^2,e^{\psi}\}$, so
\[\vert\left( w^{-1} dw + d \phi_0\right) (\sum \lambda_i V_i)\vert \le C_5 |\lambda|\]
for some constant $C_5$.   Putting all of this together with \eqref{eq:dphi0} and \eqref{eq:ftc} gives the first statement of the Lemma.  The proof of the second is exactly the same,  observing that the assumption on $\phi_0$ now means we do not need to take any further derivatives of $\psi$. 
\end{proof}

\begin{corollary}
For $|\lambda|$ sufficiently small we have
\[\widehat{\psi}^{\lambda} = \widehat{\psi} + O(1)\]
on $U\cap Y_0$.
\end{corollary}
\begin{proof}
  Apply the previous Lemma with $\phi_0 = \psi$.
\end{proof}

\begin{lemma}\label{lem:38berman}
  Suppose that $f$ is an $S^1$-invariant plurisubharmonic function on $U$ that is strictly increasing on the fibres of $Y$.  Suppose that $E'\subset X$ is locally pluripolar, and that there is a constant $c$ such that $f<c$ on $j(X-E')$ and $f>c$ on $\partial U \cap \pi^{-1}(X-E')$ where $\partial U = \{ \widehat{\phi_+} + \pi^*\psi = 1\}$.   Then there exists an extension $E(f)\in \PSH_h(Y)$ such that $E(f) =f$ on the level set $\{f=c\}$.  
\end{lemma}
\begin{proof}
  This is \cite[Lemma 3.8]{Berman} (that $E'$ is allowed to be pluripolar is remarked in the proof of the cited result).
\end{proof}

Now suppose that $g$ is some function defined on the disc bundle $U\subset Y$.  
To apply the previous result we need invariant functions, which are obtained easily through the homogenisation operator that takes a function $g$ on $U$ to
\[ H(g) = (\sup_{\theta\in [0,2\pi]} g(e^{i\theta} y) )^*\]
where the multiplication is in the fibres of $Y=L^*$.

\begin{lemma}\label{lem:extension}
Suppose that either 
\[f=   H(\widehat{\tau_C}^{\lambda})\]
or 
\[f= H\left(\frac{1}{2}(\widehat{\tau_C}^{\lambda} +  \widehat{\tau_C}^{-\lambda})\right).\]   
Then for $|\lambda|$ sufficiently small (independent of $C$) there is an extension $E(f)\in PSH_h(Y)$ such that $E(f)= f$ on the level set $S=\{y: f(y) = f(y_0)\}$.
\end{lemma}
\begin{proof}
We shall only consider $f=H(\widehat{\tau_C}^{\lambda})$ since the other case is essentially the same.  It is sufficient to construct such an extension on $Y_0$, since any such plurisubharmonic function extends to $Y$.   Now exactly as in the first half of the proof of \cite[Lemma 3.9]{Berman} one deduces that $f$ is strictly increasing along the fibres of $Y$,  since $\tau_C$ is plurisubharmonic.  

We make the normalisation so $\widehat{\tau}_C(y_0)=0$, and observe that $f(y_0)\le 1/2$ for $\lambda$ sufficiently small.   Then since $\phi_+ + \psi \le \tau_C$
\[ f \ge \widehat{\tau_C}^{\lambda} \ge \widehat{\phi_+ + \psi}^{\lambda} \ge \widehat{\phi_+} + \pi^*\psi - C'|\lambda|\]
over $U\cap Y_0$ for some $C'>0$ (independent of $C$) where we have used \eqref{lem:boundlambdafunction} with $\phi_0:= \phi_+ + \psi$ so $\phi_0 - \ln |s_{E}|^2 - \psi = \phi_A$ is smooth.  

Thus on $\partial U\cap Y_0$ we have
\[ f \ge 1 - C'|\lambda| \ge 1/2\]
for $|\lambda|$ sufficiently small (independent of $C$).  Thus the existence of the extension $\tilde{f}$ is provided by \eqref{lem:38berman} using $E'=E\cup \Sing(\psi)$. 
\end{proof}

\begin{lemma}
Suppose that $\phi$ is Lipschitz.  Then $\widehat{\phi}_{[\widehat{\psi}]}$ is Lipschitz over $Y_0$ and thus $\phi_{[\psi]}$ is Lipschitz over $X-\mathbb B_+(L-F) \cup \Sing(\psi)$.
\end{lemma}
\begin{proof}
Fix $C$ and let $\tau_C = P_{\psi+C}\phi$.   Then from the definition of the operator $H$ we have
\[\widehat{\tau_C}^{\lambda}(y_0) \le H(\widehat{\tau_C}^{\lambda})(y_0)\le EH(\widehat{\tau_C}^{\lambda})(y_0)\]
where the operator $E$ is the extension coming from Lemma \eqref{lem:extension} from the level set $S = \{y:  H(\widehat{\tau_C}^{\lambda})(y)  =  H(\widehat{\tau_C}^{\lambda})(y_0)\}$.  As $\tau_C\le \phi$ we have that on the level set $S$,
 \begin{eqnarray}
EH(\widehat{\tau_C}^{\lambda})(y) &=& H(\widehat{\tau_C}^{\lambda})(y) \le \sup_{\theta\in[0,2\pi]} \widehat{\phi}^{\lambda}(e^{i\theta}y)\\
&\le& \sup_{\theta\in[0,2\pi]} \widehat{\phi}(e^{i\theta}y) + C'|\lambda| = \widehat{\phi} + C'|\lambda|
\end{eqnarray}
for some constant $C'$ (independent of $C$) where we have used that $\phi$ is Lipschitz in the penultimate inequality to apply Lemma \ref{lem:boundlambdafunction}(2) with $\phi_0=\phi$, and the final equality comes from $S^1$ invariance of $\widehat{\phi}$ in the fibre directions of $Y$.  Since this holds on the level set $S$ we conclude by homogeneity that 

\[EH(\widehat{\tau_C}^{\lambda}) - C|\lambda| \le \widehat{\phi}\]
on all of $Y$.  

Now note that by  Lemma \ref{lem:boundlambdafunction}(1) applied with $\phi_0=\psi$ we see there is a constant $C''$ such that $\widehat{\psi}^{\lambda}\le \widehat{\psi} + C''$.   Thus using the inequality $\tau_C \le \psi'+C$ we have by similar reasoning to above
 \begin{eqnarray*}
EH(\widehat{\tau_C}^{\lambda})(y) &=& H(\widehat{\tau_C}^{\lambda})(y)\le \sup_{\theta\in[0,2\pi]} \widehat{\psi'+C}^{\lambda}(e^{i\theta}y) =\sup_{\theta\in[0,2\pi]} \widehat{\psi'}^{\lambda}(e^{i\theta}y) +C'' \\
&\le& \sup_{\theta\in[0,2\pi]} \widehat{\psi'}(e^{i\theta}y) + C'|\lambda| +C''\\
&=& \widehat{\psi'+C''}(y) + C'|\lambda|
\end{eqnarray*}
  Thus we conclude that $EH(\widehat{\tau_C}^{\lambda}) - C'|\lambda| \le \widehat{\psi'}+C''$ on $Y$.   Hence $EH(\widehat{\tau_C}^{\lambda}) - C|\lambda|$ is a candidate for the supremum appearing in the definition of $P_{\widehat{\psi'}+C''} \widehat{\phi}$ giving

\[\widehat{\tau_C}^{\lambda} - C'|\lambda| \le EH(\widehat{\tau_C}^{\lambda}) - C'|\lambda| \le  P_{\widehat{\psi'}+C''} \widehat{\phi} = \widehat{\tau_{C''}}\le \widehat{\phi}_{[\widehat{\psi}]}\]
by Lemma \ref{lemma:hatrespectenvelopes}.    Now taking the limit as $C$ tends to infinity and then the upper-semi\-continuous regularisation (which commutes with pulling back by $\theta_{\lambda}$) gives
\[ \widehat{\phi}_{[\widehat{\psi}]}^{\lambda} - C'|\lambda| \le \widehat{\phi}_{[\widehat{\psi}]}\]
Repeating the above with $\lambda$ replaced with $-\lambda$ gives the inequality
\[ | \widehat{\phi}_{[\widehat{\psi}]}^{\lambda} - \widehat{\phi}_{[\widehat{\psi}]}| \le C'|\lambda|\]
which proves that $\widehat{\phi}_{\widehat{\psi}}$ is Lipschitz near $y_0$.  Since $y_0$ was arbitrary in $U\cap Y_0$, this proves that $\phi_{[\psi]}$ is Lipschitz over  $X-\mathbb B_+(L-F) \cup \Sing(\psi)$ as claimed.
\end{proof}

\subsection{Proof of $C^{1,1}$ regularity}

\begin{theorem}\label{thm:regularity}
  Suppose that $L-F$ is big.  If $\phi$ is Lipschitz (resp.\ $C^{1,1}$) over $X$ then the same is true for $\phi_{[\psi]}$ over $X-\mathbb B_+(L-F)\cup \Sing(\psi)$.
\end{theorem}
\begin{proof}
By the above the Lipschitz statement is proved.  Thus the weak derivative of $\phi_{[\psi]}$ exists and is globally bounded.  As explained in \cite[p 21]{Berman}, to show $\phi_{[\psi]}$ is $C^{1,1}$ it is sufficient to prove the inequality
  \begin{equation}
    \label{eq:secondorderlipschitz}
    \frac{1}{2}\left(\widehat{\phi}_{[\widehat{\psi}]}^{\lambda} +  \widehat{\phi}_{[\widehat{\psi}]}^{-\lambda}\right) - \widehat{\phi}_{[\widehat{\psi}]}\le C'|\lambda|^2.
  \end{equation}
on some compact set around $y_0$.

To achieve this, argue exactly as above replacing $\widehat{\tau}_C^{\lambda}$ with 
\[ g_C = \frac{1}{2} (\widehat{\tau_C}^{\lambda}  + \widehat{\tau_C}^{-\lambda})\]  Now a simple Taylor series show that $\frac{1}{2}(  \widehat{\phi}^{-\lambda}+ \widehat{\phi}^{-\lambda})\le \widehat{\phi} + C|\lambda|^2$ over all of $X$.   The fact that $g_C\le \widehat{\psi} + C''$ is similarly shown, in fact is easier and only requires the exponentially Lipschitz hypothesis on $\psi$. 

Thus we deduce that $g_C-C'|\lambda|^2$ is a contender for the supremum defining $\widehat{\tau_C}$, so $g_C- C'|\lambda|^2\le \widehat{\tau_C}$.  Observing that $C'$ is independent of $C$, we let $C$ tend to infinity and taking the upper semicontinuous regularisation to give the desired inequality.
\end{proof}

\subsection{Monge Amp\`ere measures}
We next consider the Monge Amp\`ere measure of the maximal envelopes.  The proof of Theorem \ref{thm:extremalmeasures} is unchanged from the non-singular case, and we shall not repeat the full details of the arguments here.

\begin{lemma}
  We have
\[ D(\phi,\psi) \subset \{ x : dd^c \phi_x \ge 0\},\]
and thus ${\mathbf 1}_{D} MA(\phi) = {\mathbf 1}_{D\cap X(0)} MA(\phi)$.
\end{lemma}
\begin{proof}
This follows by observing $D(\phi,\psi) \subset D(\phi,\phi)$ and then using \cite[3.1(iii)]{Berman} and Example \ref{example:locallybounded} to deduce $D(\phi,\psi) \subset \{ x: dd^c \phi_x\ge 0\}$.    
\end{proof}

\begin{lemma}
  We have $\det(dd^c \phi_{[\psi]}) = \det(dd^c \phi)$ almost everywhere on $\mathbb B_+(L-F) \cup \Sing(\psi)$.
\end{lemma}
\begin{proof}
  The proof of this the same as \cite[p21]{Berman}.  In fact the proof given there shows that locally for any two $C^{1,1}$ metrics $\phi$ and $\phi'$ one has \[ \frac{\partial^2}{\partial z\partial \overline{z}}(\phi - \phi') =0\]
almost everywhere on the set $\{\phi = \phi'\}$.
\end{proof}

The proof of the remaining equalities
\[ \mu(\phi,\psi) = \mathbf 1_{X-\mathbb B_+(L-F)\cup \Sing(\psi)} MA(\phi_{[\psi]}) = \mathbf 1_{D} MA(\phi) = \mathbf 1_{D\cap X(0)} MA(\phi),\]
as stated in Theorem \ref{thm:extremalmeasures} is now exactly as in \cite{Berman}; we omit the details.

\subsection{The partial Bergman function}

We now turn to the partial Bergman function, and fix a smooth metric $\phi$ on $L$ and a $\psi\in \PSH(F)$ that is exponentially H\"older continuous.    We start by recalling the following upper bound on the Bergman function.  Recall $X(0) = \{ x : dd^c \phi_x>0\}$.

\begin{lemma}(Local Holomorphic Morse Inequalities)
There is a global upper bound
\[ B_k(\phi)dV \le C_k  {\mathbf 1}_{X(0)} MA(\phi)\]
where $C_k$ is a sequence of numbers that tends to $1$ as $k$ tends to infinity.
Moreover if $\Delta_k$ denotes the ball of radius $k^{-1}\ln k$ in some coordinate patch $U$ around a fixed point $x\in X$, and $f_k$ is a sequence of holomorphic functions on $U$ then
\[ \frac{|f_k(x)|^2}{\|f_k\|^2_{L^2,\Delta_k}} \le {\mathbf 1}_{X(0)}MA(\phi)k^n + o(k^n)\]
where $\| f_k\|_{L^2,\Delta_k} = \int_{\Delta_k} |f_k|^2e^{-k\phi} dV$.
\end{lemma}
\begin{proof}
  Both of these inequalities follow easily from the submean value inequality for holomorphic functions (see \cite[4.1]{Berman} or \cite[Sec 2]{Berman2}).
\end{proof}

 Recalling that $B_k(\psi)$ denotes the Bergman function of $\psi$ define the \emph{Bergman metric} as
\[\psi_k = \psi + \frac{1}{k} \ln B_k(\psi).\]

\begin{definition}\label{def:tame}
  We say that $\psi$ has \emph{tame singularities} with coefficient $c>0$ if 
\begin{equation}
   \label{eq:tame}
  \psi +O(1) \le \psi_k \le \left(1-\frac{c}{k}\right) \psi + O(1)
  \end{equation}
where the $O(1)$ term is independent of $k$.  
\end{definition}

The following condition for tameness is well known (see, for example \cite[5.10]{Boucksom2}).

\begin{lemma}\label{prop:exponentialimpliestame}
If $\psi$ is exponentially H\"older continuous with constant $c$ then it has tame singularities with constant $c^{-1}\dim X$.

\end{lemma}
\begin{proof}
Let $n=\dim X$.  Following Demailly, the Bergman metrics approximate $\psi$ in the following sense  \cite{Demaillybook}:  let $\psi$ be defined on some open ball $B$; then there is a constant $C>0$ that depends only on the diameter of $B$, such that for all $k$,
\begin{equation}
 k\psi(x) -C \le k\psi_k \le \sup_{B(x,r)} k\psi + C - n \log r\label{eq:Demailly}  
\end{equation}
where $B(x,r)$ is the ball of radius $r$ centered at $p$ with $r$ is small enough so that $B(x,r)\subset B$.    From this the lower bound in \eqref{eq:tame} follows immediately.  For the upper bound, observe first that since $\psi$ is exponentially H\"older continuous with exponent $c$,
\[ \sup_{B(x,r)} e^{\psi} \le e^{\psi(x)} + C_1 r^c\]
for some constant $C_1$ and thus
\[ \sup_{B(x,r)} \psi \le \ln (1+C_1) + \psi(x).\]
  Moreover, from the same assumption, one sees there is a $C_2$ such that if $r=C_2e^{\psi(x)/c}$ then $B(x,r)\cap \Sing(\psi)$ is empty.    Hence applying \eqref{eq:Demailly} with this value of $r$ yields
\[ k\psi_k(x) \le k\ln(1+C_1) + C + (k-nc^{-1}) \psi(x) -n\log C_2\]
as required.
\end{proof}

\begin{proposition}\label{proposition:bergmanboundtame}
Suppose $\psi\in \PSH(F)$ has tame singularities with constant $c$.  There is a constant $C$ depending on $\phi$ and $\psi$ such that
\[ B_k(\phi,\psi) \le C k^n e^{-c\psi} e^{k(\phi_{[\psi]} - \phi)}\]  
over all of $X$.  In particular
\[\lim_{k\to \infty} k^{-n} B_k(\phi,\psi) =0 \text{ for }x\notin D(\phi,\psi)\cup\Sing(\psi).\]
\end{proposition}
\begin{proof}
As is easily verified, the partial Bergman kernel has the extremal property
\[ B_k(\phi,\psi) = \sup \{ |s(x)|_{\phi}^2 : \|s\|_{\phi,dV}=1, s\in H^0(\mathcal I(k\psi) L^k)\}.\]
Thus it is sufficient to prove the existence of a $C$ such that if $s\in H^0(\mathcal I(k\psi) L^k)$ and $\|s\|_{\phi,dV} = k^{-n}$ then 
\begin{equation*}
 |s|^2_{\phi} \le C e^{-c\psi}e^{k(\phi_{[\psi]} - \phi)}.  
\end{equation*}
  From the holomorphic Morse inequalities there is a $C$ with $k^{-n}B_k(\phi)\le C$ for all $k$.    So since $|s|_{\phi}^2 \le k^{-n}B_k(\phi,\psi)\le k^{-n}B_k(\phi)$ we have
  \begin{equation}
    k^{-1} \ln |s|^2 - k^{-1} \ln C \le \phi.    
  \end{equation}
Moreover by the assumptions on the singularities of $\psi$,
\[ k^{-1} \ln |s|^2 \le \psi_k + O(1) \le (1-k^{-1}c) \psi + O(1) \]
where the $O(1)$ term may depends on $k$.   Without loss of generality suppose $\psi\le 0$ globally.  Then we see $k^{-1} \ln |s|^2 + ck^{-1} \psi - k^{-1} \ln C$ is bounded above by both $\phi$ and $\psi+O(1)$, and thus is a candidate for the supremum defining $P_{[\psi]}\phi$, so
\[k^{-1} \ln |s|^2 + ck^{-1} \psi - k^{-1} \ln C \le \phi_{[\psi]}\]
and rearranging proves the first statement of the proposition, from which the second follows immediately.
\end{proof}

Before moving on we observe that a slightly more precise statement is possible when $\psi$ has algebraic singularities.  Suppose that the singularities of $\psi$ is modeled on $(\mathcal I,c)$.  We fix a resolution $\pi\colon \tilde{X}\to X$ such that $\pi^* \mathcal I = \mathcal O(-D)$ where $D=\sum_j \alpha_j D_j$ is a normal crossing divisor (see for example \cite[5.9]{Demaillybook} for this basic technique).  

\begin{definition}\label{def:potentialjumping}
The set of \emph{potential jumping numbers} for $\psi$ is
\[J(\psi) = \{ k  : k\alpha_j \in \mathbb N \text{ for all } j\}.\]
\end{definition}

  In the simplest case,  $\psi = \sum_j \alpha_j \ln |g_j|$ globally with
  $\alpha_j\in \mathbb Q^+$ in which case it is clear that $J(\psi) = \{ k : k\alpha_j \in
  \mathbb N \text{ for all } j\}$.  It is clear in general that potential jumping numbers exist that are arbitrarily large.   The terminology comes from the fact that $J(\psi)$ restricts the set on which the multiplier ideals $\mathcal I(t\psi)$ can ``jump'' as $t\in \mathbb R^+$ varies.

\begin{proposition}\label{proposition:bergmanboundalgebraic}
There is a constant $C$ (depending on $\phi)$ such that for all $\psi$ with algebraic singularities, and all $k\in J(\psi)$ we have
\[ B_k(\phi,\psi) \le C k^n e^{k(\phi_{[\psi]} - \phi)}\]  
\end{proposition}
\begin{proof}
We first consider first the special case that $\psi$ is of the form $\psi=\sum_j \alpha_j \ln |g_j|^2$ with $D_j=g_j^{-1}(0)$ smooth normal crossing divisors.  As in the previous proof we have to show that if $s\in H^0(\mathcal I(k\psi) L^k)$ with $k\in J(\psi)$ and $\|s\|_{\phi,dV} = k^{-n}$ then $|s|_{\phi}^2\le Ce^{k(\phi_{[\psi]} - \phi)}$.  Just as above, from the holomorphic Morse inequalities we have $k^{-1}\ln |s|^2 - k^{-1} \ln C\le \phi$ where $C$ is independent of $\phi$.       Now as $s$ lies in the multiplier ideal, locally we can write
\[ s = h \Pi_{j} g_j ^{2m_j}\]
for some holomorphic $h$ and $m_j \ge \lfloor k\alpha_j\rfloor = k\alpha_j$ \cite[5.9]{Demaillybook}
(here we have used $k\in J(\psi)$ so $k\alpha_j\in \mathbb N$).  Thus
\[ |s|^2 \le C' \Pi_j |g_j|^{2m_j}\]
for some constant $C'$ (which may depend on $s$).   Therefore we deduce that near $D=\Sing(\psi)$,
\[ k^{-1} \ln |s|^2 \le k^{-1} \ln C' + k^{-1} \sum_j m_j \ln |g_j|^2 \le \sum_j \alpha_j \ln |g_j|^2 + C''=  \psi + C''.\]

Thus $k^{-1} \ln |s|^2 - k^{-1} \ln C$ is a candidate for the supremum defining $\phi_{[\psi]}$ which gives the result for this particular form of $\psi$.

Now for general $\psi$ with algebraic we can reduce to the above by passing to log resolution $\pi\colon \tilde{X}\to X$.   As is easily checked, if $\tilde{\psi} = \sum_j \alpha_j \ln |g_j|^2$ where $D_j = g_j^{-1}(0)$ then $\pi^* B_k(\phi,\psi) = B_k(\pi^*\phi, \tilde{\psi})$ and $\pi^* \psi \sim \tilde{\psi}$, giving $\pi^* (\phi_{[\psi]}) = (\pi^* \phi)_{[\tilde{\psi}]}$.   Then applying the previous part of the proof gives the result we require.
\end{proof}

We now turn to proving a lower bound for the partial Bergman function.    Recall that $\phi_F$ is a fixed smooth metric on $F$ and we have set $\psi' = \phi - \phi_F + \psi$.

\begin{proposition}\label{prop:modify}
Let $\omega$ be a K\"ahler form on $X$.    There exists a $k'$ and $C'>0$ such that for any $k\ge k'$, any $x\in X - \mathbb B_+(L-F)\cup \Sing(\psi)$ and any $C>0$ there exists a metric $\sigma_k=\sigma_{k,C}$ on $L^k$ such that
\begin{enumerate}
\item $\sigma_{k,C} \le k P_{\psi'+C}\phi$ on $X$
\item There is an neighbourhood $U$ of $x$ and constant $C'_U$ (both independent of $k$ and $C$) such that 
\[kP_{\psi'+C} \phi  \le \sigma_k+C_U'\text{ on }U.\]
\item $\sigma_{k,C}$ has the extension property with constant $C'$ (as defined in Section \ref{sec:ohsawa}).
\item $dd^c\sigma_{k,C} \ge C'^{-1}\omega$ for all $k$ and $C$.
\end{enumerate}
\end{proposition}

\begin{proof}
We may as well assume $L-F$ is big otherwise the statement is trivial.  Thus $L-F$ admits a metric $\phi_+$ with strictly positive curvature that is smooth away from $\mathbb B_+(L-F)$.   By subtracting a constant from $\phi_+$ we may assume that
\begin{equation}
\phi_+ \le \min\{\phi-\psi,\phi-\phi_F\}\label{eq:metricbound}
\end{equation}
which yields $\phi_+ + \psi \le P_{\psi'} \phi \le P_{\psi'+C}\phi$ for all $C>0$.

Thus there is a $c>0$ with $dd^c \phi_+\ge c\omega$, so we can take $k'$ and $C'$ as in the statement of the Ohsawa-Takegoshi Theorem \eqref{thm:ohsawa}.  Now let $k\ge k'$ and $x\in X -\mathbb B_+(L-F)\cup \Sing(\psi)$.  Set 
\[\phi_{F'} = (k-k') P_{\psi'+C} \phi + k'\psi\]
which is a metric on $F' = L^{k-k'} \otimes F^{k'}$ and let
\[\sigma_k = k'\phi_+ + \phi_{F'}.\]
Observe that $dd^c \phi_+\ge c\omega$ and $dd^c\phi_{F'}\ge 0$, so the extension property (3) holds and also $dd^c\sigma_{k,c}\ge k'dd^c\phi_+$ so (4) follows immediately.  Moreover 
\[\sigma_k = k'(\phi_+ + \psi) + (k-k')P_{\psi'+C} \phi \le kP_{\psi'+C}\phi\]
which gives (1).   Finally
\begin{eqnarray*}
  kP_{\psi'+C} - \sigma_k& =& k'(P_{\psi'+C}-\psi-\phi_+)\\
&\le & k'(\phi -\psi - \phi_+)
\end{eqnarray*}
Hence if $U$ is a small ball around $x$ in $\mathbb B_+(L-F)\cup \Sing(\psi)$ then $\phi-\psi-\phi_+$ is bounded on $U$ by some constant $C'_U$, so
$kP_{\psi'+C} - \sigma_k \le C'_U$ on $U$ for all $C>0$.
\end{proof}

\begin{theorem}\label{thm:bergmanconverge}
We have
  \[k^{-1} \ln B_k \to \phi_{[\psi]}-\phi\]
  uniformly on compact subsets of $X-\mathbb B_+(L-F)\cup \Sing(\psi)$ as $k$ tends to infinity.    That is, given a compact subset $K\subset X-\mathbb B_+(L-F)\cup \Sing(\psi)$ there is a $C_K>0$ such that
\[ C_K^{-1} e^{k(\phi- \phi_{[\psi]})} \le B_k(\phi,\psi) \le C_K k^n e^{k(\phi- \phi_{[\psi]})}\]
over $K$ for all $k$.
\end{theorem}
\begin{proof}
  The upper bound for $B_k$ comes from Proposition \ref{proposition:bergmanboundtame} since $\psi$ is bounded on any compact $K$ outside of $\Sing(\psi)$.      For the other direction let $k'$ and $C'$ be as in the statement of \ref{prop:modify} and $x\in X-\mathbb B_+(L-F)\cup \Sing(\psi)$.  To ease notation let $\tau_C = P_{\psi'+C}\phi$.  Pick a metric $\sigma_k$ as in the previous proposition.  So by the extension property of $\sigma_k$ there exists an $s\in H^0(L^k)$ such that $\|s\|_{\sigma_k}\le C'$ and $|s(x)|_{\sigma_k} = 1$.   But this implies that $\|s\|_{k\tau_C} \le C'$ and $|s(x)|_{k\tau_C}\ge (C'_U)^{-1}$.  In particular $\|s\|_{k(\psi'+C)}$ is finite, and so $s\in H^0(\mathcal I(k\psi) L)$.  Thus by the extremal property of the Bergman function we deduce that
\[ B_k(\phi,\psi)\ge C' e^{k(\tau_C - \phi)}\]
on some neighbourhood $U$ of $x$, for some constant $C'$ that is independent of $C$ and $U$.  Letting $C$ tend to infinity and then taking the upper-semicontinuous regularisation (using that the Bergman function is continuous) yields
\[ B_k(\phi,\psi) \ge C' e^{k(\phi_{[\psi]} - \phi)}\]
on $U$ as required.
\end{proof}

\begin{lemma}\label{lemma:lowerbergman}
We have
\[ \liminf_{k\to \infty} k^{-n}B_k(\phi,\psi)dV\ge MA(\phi)\]
almost everywhere on $D(\phi,\psi)\cap X(0).$
\end{lemma}
\begin{proof}

It is sufficient to prove the existence of a sequence $s_k\in H^0(L^k\otimes \mathcal I(k\psi))$ such that
  \begin{eqnarray}\label{eq:lowerbergman0}
    \lim_{k\to \infty} \frac{|s_k(x)|_{\phi}^2}{k^n\|s_k\|^2_{\phi}} dV &=& MA(\phi)_x
  \end{eqnarray}

Since $x$ is general we may assume $x\notin \mathbb B_+(L-F)\cup \Sing(\psi)$.  So from Lemma \ref{prop:modify} there is an open set $U$ around $x$ and metrics $\sigma_{k,C}\in \PSH(L^k)$ for $k\ge 0$ and $C>0$ such that 
\begin{eqnarray}\label{eq:lowerbergman1}
  \sigma_{k,C}\le kP_{\psi'+C} \text{ on }X,\\
  kP_{\psi'+C} \le C_1 + \sigma_{k+C}\text{ on }U\text{ and},\nonumber\\
dd^c\sigma_{k,C} \ge C_1^{-1}\omega,\nonumber
\end{eqnarray}
where the constant $C_1$ is independent of $k$ and $C$ and $\omega$ is some chosen smooth K\"ahler form.

Next we consider smooth sections of $L^k$ that are ``peaked'' and supported in $U$.  Using cut-off functions, and that $\phi_{[\psi]}$ is $C^{1,1}$, one can produce,  for sufficiently general points $x$ in $D(\phi,\psi)\cap X(0)$, a sequence of smooth sections $f_k$ of $L^k$ that are supported in $U$ and satisfy
  \begin{eqnarray*}
    \lim_{k\to \infty} \frac{|f_k(x)|_{\phi}^2}{k^n\|f_k\|^2_{\phi}} dV &=& MA(\phi)_x\\
\|\bar{\partial} f_k\|_{k\phi_{[\psi]}} &\le& C_2e^{-k/C_2}
  \end{eqnarray*}
for some constant $C_1$ (see \cite[Lemma 4.4]{Berman}).  

To perturb these to holomorphic sections we apply the H\"ormander estimate with the metric $\sigma_{k,C}$ to obtain, for $k$ sufficiently large,  smooth sections $g_{k,C}$ of $L^k$ with $\bar{\partial} g_{k,C} = \bar{\partial} f_k$ and $\|g_{k,C}\|_{\sigma_k,C} \le C_3 \|\bar{\partial} f_k\|_{\sigma_{k,C}}$ (observe $C_3$ can be taken independent of $C$ and $k$ from the lower bound for $dd^c\sigma_{k,C}$ from \eqref{eq:lowerbergman1}).    Now the first two statements in \eqref{eq:lowerbergman1}, and the fact that $f_k$ is supported on $U$ imply
\[ \|g_{k,C}\|_{k\phi}\le \|g_{k,C}\|_{kP_{\psi'+C}\phi} \le \|g_{k,C}\|_{\sigma_{k,C}} \le C_3 \|\bar{\partial} f_k\|_{\sigma_{k,C}} \le C_4 \|\bar{\partial} f_k\|_{kP_{\psi'+C}},\]
where $C_4$ is also independent of $C$ and $k$.  

Temporarily fixing $k$, and recalling that $P_{\psi'+C}$ tends to $\phi_{[\psi]}$ pointwise almost everywhere as $C$ tends to infinity, an application of the dominated convergence theorem yields
\[\lim_{C\to \infty} \|g_{k,C}\|_{k\phi}\le C_4 \|\bar{\partial} f_k\|_{k\phi_{[\psi]}}\le C_5 e^{-k/C_2}.\]
In particular choosing $C= C(k)$ sufficiently large we deduce that for each $k$ there is an $h_k = g_{k,C(k)}$ such that $\bar{\partial} h_k= \bar{\partial} f_k$ and
\begin{equation}
  \label{eq:lowerbergman2}
 \|h_k\|_{\phi}\le 2C_5 e^{-k/C_2}.  
\end{equation}

Thus $s_{k}:= f_k-h_k$ is holomorphic and since $\|s_{k}\|_{kP_{\psi'+C}}$ is finite and $P_{\psi'+C}\le \psi +C$ we in fact have $s_k\in H^0(L^k\otimes \mathcal I(k\psi')) = H^0(L^k\otimes \mathcal I(k\psi))$.    Moreover $\|s_k\|^2_{\phi}\le \|f_k\|^2_{\phi} + O(k^{-\infty})$ and, since $h_k$ is holomorphic near $x$, the local holomorphic Morse inequality implies $|s_k(x)|_{\phi}^2= |f_k(x)|^2_{\phi} + O(k^{-\infty})$ as well.  Thus \eqref{eq:lowerbergman0} holds for this sequence of sections $s_k$, completing the proof of the Lemma.
\end{proof}

\begin{theorem}\label{thm:bergmanmeasurelimit}
Suppose $L-F$ is big.  Then there is a pointwise limit 
\[ \lim_{k\to \infty} k^{-n} B_k(\phi,\psi) dV = \mathbf 1_{D(\phi,\psi)\cap X(0)} MA(\phi)\]
almost everywhere on $X(0)$.  Moreover \[\lim_{k\to \infty} k^{-n}B_k(\phi,\psi)dV \to \mu(\phi,\psi)\] weakly in the sense of measures.
\end{theorem}
\begin{proof}
  We have shown that $\lim_{k\to \infty} k^{-n} B_k(\phi,\psi)=0$ for $x\notin D(\phi,\psi)\cup\Sing(\psi)$, and thus almost everywhere outside $D(\phi,\psi)$ since $\Sing(\psi)$ is pluripolar and thus has measure zero.    For general $x\in D(\phi,\psi)$ the limit follows by combining the upper bound coming from the local holomorphic Morse inequalities, and the lower bound in the previous proposition.  The statement about the measures now follows from the Dominated Convergence Theorem, and the global upper bound $k^{-n}B_k(\phi,\psi)\le k^{-n}B(\phi)\le C$.
\end{proof}

In particular, by integrating the previous theorem over $X$ we see the volume of the equilibrium set $D(\phi,\psi)$ captures the rate of growth of the filtration of the space of sections determined by the multiplier ideal of $\psi$:
\begin{corollary}\label{cor:equilibriumvolume}
  \[\int_{D(\phi,\psi)} MA(\phi) = \lim_{k\to \infty} k^{-n} h^0(L^k \otimes \mathcal I(k\psi)).\]
\end{corollary}

\section{Maximal Envelopes on Products}\label{sec:products}

In this section we shall consider maximal envelopes on products.  Suppose  that $X_i$ for $i=1,2$ are smooth complex manifolds on which we have line bundles $L_i$ with smooth metrics $\phi_{i}$.  Our aim is to consider maximal envelopes on the product $X_1\times X_2$ with respect to the product metric \[\phi:=\pi_1^{*}\phi_{1} + \pi_2^{*}\phi_{2}\] 
where $\pi_i$ are the projection maps (for simplicity we shall suppress the $\pi_i$ in what follows where it cannot cause confusion).

Let $F_i$ be additional line bundles on $X_i$ with metrics $\psi_i\in \PSH(F_i)$.  For simplicity we assume that $L_i-F_i$ are ample, so $\mathbb B_+(L_i-F_i)$ is empty.  Furthermore set
\[\psi' = \sup\{\psi_1,\psi_2\}\]
so, recalling the abuse of notation described in \eqref{eq:abusesupremum}, $\psi'\in \PSH(F_1+F_2)$.

\begin{theorem}\label{thm:maximalenvelopeproduct}
Suppose $\psi_i$ have algebraic singularities.  Then
  \begin{eqnarray}\label{eq:envelopeproduct}
 \phi_{[\psi']} &=& \sup\{ (\phi_{1})_{[\lambda\psi_1]} + (\phi_{2})_{[(1-\lambda)\psi_2]} : \lambda \in (0,1)\}^*
\end{eqnarray}
\end{theorem}

Said another way, such maximal envelopes on products can be calculated by considering only metrics whose variables separate, i.e.
   \begin{eqnarray*}
     \phi_{[\psi']} &=& \sup\{ \gamma_1+\gamma_2, \gamma_i\in PSH(L_i), \gamma_1\sim \lambda \psi_1, \gamma_2\sim (1-\lambda) \psi_2,\lambda\in (0,1)\}^*.
\end{eqnarray*}

As remarked in the introduction, such a formula resembles known results for the Siciak extremal function on products of domains in $\mathbb C^n$ (see, for example, \cite{Bedford4, Blocki, Sickiak}) and for the pluricomplex Green function on products \cite{Rashkovskii} and these results may suggest other methods of proof.  The techniques we employ here have an algebraic flavour, using results from the previous section to recast the problem in terms of the partial Bergman function and then applying a combination of the K\"unneth formula and the Musta\c{t}\u{a} summation formula for multiplier ideals.  

\begin{proof}[Proof of Theorem \ref{thm:maximalenvelopeproduct}]
Fix $\lambda\in (0,1)$ and suppose $\gamma_i\in PSH(L_i)$ with $\gamma_i\le \phi_{i}$ for $i=1,2$ and $\gamma_1\sim \lambda\psi_1$ and $\gamma_2\sim (1-\lambda)\psi_2$.  Thus there is a constant $C$ such that $\gamma_i\le \lambda \psi_1+C$ and $\gamma_2\le (1-\lambda)\psi_2+C$, which yields
\[\gamma_1 + \gamma_2 \le \lambda \psi' +C + (1-\lambda) \psi'+C = \psi' +2C.\]
Hence $\gamma_1+\gamma_2\sim \psi'$ which implies $\gamma_1+\gamma_2\le \phi_{[\psi']}$.  Taking the supremum over all $\gamma_1$ and then all $\gamma_2$ yields
\[ (\phi_1)_{[\lambda \psi_1]} + (\phi_2)_{[(1-\lambda)\psi_2]}\le \phi_{[\psi']}.\]
Thus taking the supremum over all $\lambda\in (0,1)$ shows the right hand side of \eqref{eq:envelopeproduct} is less than or equal to the left hand side.

Now as $\psi_i$ have algebraic singularities, it is immediate that \eqref{eq:envelopeproduct}  holds on $\Sing(\phi_{[\psi']}) = \Sing(\psi') = \Sing(\psi_1) \times \Sing(\psi_2)$.  So suppose $x\in X_1\times X_2$ is such that $\phi_{[\psi']}(x)\neq -\infty$ and let  $\epsilon>0$. Writing $L=L_1\otimes L_2$, we know from Theorem \ref{thm:partialbergman} that for $k$ sufficiently large there exists an $s\in H^0(\mathcal I(k\psi') L^k)$ with $\|s\|_{\phi}=1$  and
\[ k^{-1} \ln |s(x)|_{\phi}^2 \ge \phi_{[\psi']}(x) - \epsilon.\]

Consider next the multiplier ideals $\mathcal I(tk\psi)$ for $t\in \mathbb R^+$.  These form an nested sequence of ideal sheaves that induce a finite filtration
\[ 0 =H^0(\mathcal I(t_1k\psi L_1^k) \subset H^0(\mathcal I(t_2k\psi) L_1^k) \subset\cdots \subset H^0(\mathcal I(t_Nk\psi) L_1^k)=H^0(L_1^k).\]

Pick an $L^2$-orthonormal basis $\{a_j\}$ for $H^0(L_1^k)$ that is compatible with this filtration, so $a_j \in H^0(\mathcal I(k\lambda_j\psi_1) L_1^k))$.    Similarly pick an $L^2$-orthonormal basis $\{b_j\}$ for $H^0(L_2^k)$ so $b_j\in H^0(\mathcal I(\mu_jk\psi_2)L_2^k)$.   Then since $H^0(L) = H^0(L_1^k) \otimes H^0(L_2^k)$ we can write
\[ s = \sum_{ij} \alpha_{ij} a_j \otimes b_j\]
where $\sum_{ij} |\alpha_{ij}|^2 =1$.   Now since all the metrics in question have algebraic singularities, the Musta\c{t}\u{a} summation formula \cite{Mustata} gives
\[\mathcal I(k\psi') \subset  \sum_{\lambda + \mu = 1 } \mathcal I(k\lambda \psi'_1) \mathcal I(k\mu \psi'_2).\]  
Thus we deduce $\alpha_{ij}=0$ unless $\lambda_j + \mu_j = 1$.

Fix $(i_0,j_0)$ so $|a_{i_0} \otimes b_{j_0}(x)| \ge |a_i\otimes b_j(x)|$ for all $i$ and $j$.  Hence if $N_k = h^0(L^k)$ we have $|s(x)|^2\le N_k |a_{i_0}|^2 |b_{j_0}|^2$. Define $\gamma_1 = k^{-1} \ln |a_{i_0}|^2$ and $\gamma_2 = k^{-1} \ln |b_{j_0}|^2$ so $\gamma_i\in PSH(L_i)$, and
\[ \phi_{[\psi']}(x) -\epsilon \le k^{-1} \ln |s(x)|^2 \le k^{-1} \ln N_k + \gamma_1(x) + \gamma_2(x).\]
 As $\psi_1$ is algebraic it is certainly tame (with constant $c$ say).  Thus from \eqref{proposition:bergmanboundtame} we have
\begin{eqnarray*}
 \gamma_1(x)&\le& (\phi_1)_{[\lambda_j \psi_1]}(x) + k^{-1} \ln C - k^{-1} c \lambda_j \psi_1(x) \le (\phi_1)_{[\lambda_j \psi_1]}(x)  + k^{-1}\ln C'
 \end{eqnarray*}
since $\lambda_j\le 1$ and with a similar expression bounding $\gamma_2(x)$.  Now $N_k$ is bounded by a polynomial in $k$, for $k$ sufficiently large,
\begin{eqnarray*}
 \phi_{[\psi']}(x)- \epsilon&\le& (\phi_1)_{[\lambda_j \psi_1]} + (\phi_2)_{[\mu_j \psi_2]} + \epsilon \\
&\le& \sup_{\lambda\in (0,1)} \{ (\phi_1)_{[\lambda \psi_1]} +  (\phi_2)_{[(1-\lambda) \psi_2]} \}^*+ \epsilon.
\end{eqnarray*}
Since $\epsilon$ was arbitrary this gives the inequality required.  
\end{proof}

We do not expect the previous theorem to be optimal.  In addition to the likelihood of being able to relax the assumptions on the singularity type of $\psi_i$, it seems reasonable to conjecture that an analogous statement holds for maximal envelopes coming from ``test curves'' of singularities.   In the simplest case, where $F=L$,  an example of a test curve is a family $\psi_{\lambda}\in \PSH(L)$ for $\lambda\in (0,1)$ that is concave in $\lambda$ (see below \eqref{definitiontestcurve} for more general definition).

 \begin{conjecture}\label{productconjecture}
Suppose $\psi_{1,\lambda}$ and $\psi_{2,\lambda}$ are test curves for $L_1$ and $L_2$ respectively and set $\psi = \sup\{\psi_{1,\lambda} + \psi_{2,1-\lambda}\}^*.$  Then, possibly under some regularity assumptions of the test curves, the maximal envelope on the product is given by
\[    \phi_{[\psi]} = \sup \{  (\phi_1)_{[\psi_{1,\lambda}]} + (\phi_2)_{[\psi_{2,1-\lambda}]}  : \lambda\in (0,1) \}^*.\]
 \end{conjecture}

We have not seriously attempted to prove this conjecture and thus will not discuss it much further, but presumably the simplest next case to consider are ``piecewise linear'' test curves that are locally in $\lambda$ of the form $\psi_{\lambda} = \zeta_0 + \lambda \zeta_1$ for fixed singular metrics $\zeta_0$ and $\zeta_1$.   It would be interesting also to investigate if this generalisation has an algebraic counterpart being related to some kind of ``limit'' of the Musta\c{t}\u{a} summation formula.

\section{The Legendre Transform as a Maximal Envelope}\label{sec:legendre}

Our goal in this section is to show how maximal envelopes on the product of $X$ with a disc captures the Legendre transform of a test curve of singular metric, as previously considered by the authors \cite{RossNystrom}.  In the following fix a compact complex $X$ and big line bundle $L$ and some smooth positive metric $\phi$ on $L$.

\begin{definition}\label{definitiontestcurve}
Suppose $F$ is a line bundle on $X$.  A \emph{test curve} on $F$ is a map $\lambda \mapsto \psi_{\lambda} \in \PSH(\lambda F)$ for $\lambda \in (0,c)$ for some $c$ such that
\begin{enumerate}
\item $\psi_\lambda$ is concave in $\lambda$
\item $\lambda^{-1}\psi_{\lambda}$ is decreasing in $\lambda$
\item There is a metric $\phi_F$ on $F$ such that $\lambda^{-1} \psi_\lambda \le \phi_F$ for all $\lambda\in (0,c)$.
\end{enumerate}
\end{definition}

\begin{example}
  The simplest example is when $F$ has a holomorphic section $s$ in which case $\psi_{\lambda} = \lambda \ln |s|^2$ for $\lambda\in (0,1)$ defines a test curve on $F$.
\end{example}

\begin{remark}
This definition differs slightly from that used \cite{RossNystrom}.  To see the compatibility, suppose that $L$ is ample and set $F=L$.  Then let $\psi_{\lambda}\in \PSH(\lambda L)$ for $\lambda\in (0,1)$ be a test curve on $L$.   By definition, we can pick a positive $\phi_L$ on $L$ such that $\phi_L\ge \lambda^{-1}\psi_{\lambda}$ for all $\lambda$.    Now set 
\[\psi'_{\lambda} =
\left\{
  \begin{array}{ll}
\phi_L & \lambda\le 0\\
\psi_{\lambda} + (1-\lambda) \phi_L& \lambda\in(0,1)    \\
-\infty & \lambda\ge 1
  \end{array}\right. 
\]
Then (1) $\psi'_{\lambda}\in PSH(L)$ and is concave in $\lambda$, (2) $\psi'_{\lambda}$ is locally bounded for $\lambda<0$ and (3) $\psi'_{\lambda}=-\infty$ for $\lambda>1$.    If we assume in addition  $\psi_{\lambda}$ has small unbounded locus for $\lambda\in(0,1)$ then this is essentially what is called test curve in the sense of \cite{RossNystrom}.  Notice furthermore that for $\lambda\in (0,1)$ we have that $L -\lambda F = (1-\lambda) L$ is big (it is even ample) and
$$\phi_{[\psi'_{\lambda}]} = \phi_{[\psi_{\lambda}]}.$$
Thus we may apply our regularity result, Theorem \ref{thm:regularityintroduction}, and its consequences, to these envelopes.
\end{remark}

Now fix a test curve $\psi_{\lambda}$, and set
\[ \phi_{\lambda} = \phi_{[\psi_{\lambda}]} \text{ for } \lambda\in(0,c).\]
It is convenient also to define $\phi_{\lambda} = \phi$ for $\lambda\le 0$ and $\phi_{\lambda} = -\infty$ for $\lambda\ge c$.  

\begin{definition}
  The \emph{Legendre transform} of a test curve $\psi_{\lambda},$
  is defined to be
  \[\widehat{\phi}_t:=(\sup_{\lambda\in
    \mathbb{R}}\{\phi_{\lambda}+t\lambda\})^*,\] where
  $t\in[0,\infty).$
\end{definition}

In \cite{RossNystrom} the authors prove that, when $L$ is ample, the Legendre transform is a weak geodesic ray in the space of metrics in $\PSH(L)$ emanating from $\phi$.  By this it is meant that  $t=-\ln |z|$ where $z$ is a coordinate on the closed unit disc $B$ in $\mathbb C$, then $\Phi(x,t) := \widehat{\phi}_t(x)$ defines a positive metric on the pullback of $L$ to $X\times B$ that is $S^1$ invariant and solves the Homogeneous Monge Amp\`ere equation $MA(\Phi)=0$.    
\medskip

We proceed now to show how the Legendre transform itself a maximal envelope.  Let $\pi\colon X\times B\to X$ be the projection and set
\[ \psi' =\sup_{\lambda\in (0,c)} \{ \psi_{\lambda} + \lambda t\}^*\]
where the $c$ is as in the definition of the test curve.

\begin{theorem}
Let $\phi' = \phi + ct$.  Then over $X\times B$ we have
  \[\phi'_{[\psi']} = \widehat{\phi_t}\]
\end{theorem}

\begin{proof}
Suppose $\gamma\in PSH(L)$ with $\gamma\le \phi$ and $\gamma\le \psi_{\lambda} +C$ for some $\lambda\in (0,c)$ and constant $C$.  Then as $t\ge 0$,  $\gamma + \lambda t \le \gamma + ct\le \phi + ct = \phi'$ and
\[ \gamma + \lambda t \le \psi_{\lambda} + \lambda t + C \le \psi' + C.\]
Hence $\gamma + \lambda t$ is a candidate for the envelope $P_{\psi'+C} \phi'$ giving
\[ \gamma + \lambda t \le P_{\psi'+C} \phi'\le \phi'_{[\psi']}.\]
Taking the supremum over all such $\gamma$ gives
\[ P_{\psi_{\lambda}+C} \phi + \lambda  t \le \phi'_{[\psi']},\]
so taking the limit as $C$ tends to infinity and then the upper semicontinuous regularisation yields
\[\phi_{\lambda} + \lambda t  \le  \phi'_{[\psi']}\]
and then taking the supremum over all $\lambda\in (0,c)$ gives
\[\widehat{\phi}_t \le \phi'_{[\psi']}.\]

For the other inequality, for fixed $C>0$ and $\lambda\in (0,c)$ define
\[\gamma^C_{\lambda}(x) = \inf_{t\ge 0} \{ P_{\psi'+C} \phi'(x,t) + \lambda  t\}.\]
Thus $\gamma^C_{\lambda}$ is an infimum of plurisubharmonic functions on $X\times B$ that depends only on the modulus of $w\in B$, so by the Kiselman minimum principle $\gamma_{\lambda}$ is plurisubharmonic.    Now one clearly has $\gamma^C_{\lambda}(x)\le \phi'(x,0)= \phi(x)$ and moreover
\[ \gamma^C_\lambda \le \inf_{t\ge 0} \{\psi' +C + \lambda t \} = \psi_{\lambda} + C\]
as the Legendre transform is an involution.  Hence, $\gamma_{\lambda}^C\le \phi_{[\psi_{\lambda}]}$ and so
\[\gamma^C_{\lambda} + \lambda  t \le  \widehat{\phi}_t.\]
Now taking the supremum over all $\lambda\in (0,c)$ and using the involution property of the Legendre transform again yields
\[ P_{\psi'+C}\phi' \le \widehat{\phi}_t.\]
Taking $C$ to infinity and the upper semicontinuous regularisation gives 
\[ \phi'_{[\psi']} \le \widehat{\phi}_t\]
which completes the proof.
\end{proof}

\begin{remark}
  As previously remarked, the previous theorem is in fact a special case of the general conjecture we made for the maximal envelope of a product \eqref{productconjecture}.  Essential in the above proof is that all quantities defined on $B$ have been taken to be $S^1$ invariant (i.e.\ they depend only on $t = -\ln |w|$ rather than $w$) and so the Kiselman minimal principle can be applied.  
\end{remark}

\begin{remark}\label{rmk:regularity}
  As a consequence of the previous theorem, clearly any regularity enjoyed by the maximal envelope $\phi'_{[\psi']}$ will be similarly enjoyed by the geodesic $\widehat{\phi}_t$.   A simple modification of the argument in the previous section shows that, as long as $\psi'$ is exponentially H\"older continuous and $\phi$ is smooth, then $\widehat{\phi}_t$ is $C^{1,1}$ in the direction of $X$ (the modification needed is simply to replace $Y$ with $L^*\times B$ and demand that the vector fields in \eqref{lem:existencevectorfields} lie in the subbundle $TL^*\subset T(L^*\times B)$).   Moreover one can show with these same techniques that $\widehat{\phi}_t$ is Lipschitz in the $t$ variable, but we have not been able to use them to show it is $C^{1,1}$ in this direction as well.  
\end{remark}

\section{Exhaustion functions of Equilibrium sets}\label{gradientmaps}

We continue to consider the maximal envelopes
\[\phi_{\lambda} := \phi_{[\psi_{\lambda}]}\]
where $\psi_{\lambda}$ is a test curve of singularities, and the associated equilibrium sets
\[D_{\lambda}:= D(\phi,\psi_{\lambda}) = \{ x\in X : \phi_{\lambda} = \phi\}.\]
Observe that $\phi_{\lambda}$ are decreasing in $\lambda$, and thus $D_{\lambda}$ are closed and increasing, i.e.\
\[ D_{\lambda}\subset D_{\lambda'} \quad \text{ if } \quad \lambda\ge \lambda'.\]

\begin{definition}
  Denote the exhaustion function  $H\colon X\to \mathbb R$ by
  \begin{eqnarray*}
  H(x) &=& \sup\{ \lambda>0 : x\in D_\lambda\}\\
&=&\sup\{ \lambda>0 : \phi_{\lambda}(x) = \phi(x)\}.
\end{eqnarray*}
\end{definition}
When necessary we shall write $H_{\phi,\psi}$ or $H_{\phi}$ to emphasise the
dependence on the metrics in question.  Note that since each $D_{\lambda}$ is closed, it is clear that $H_{\phi}$ is upper-semicontinuous.\medskip

We show now that this exhaustion function is the time derivative of the associated geodesic ray coming from the Legendre transform:

\begin{theorem}\label{thm:rightderivative2}
Fix a smooth metric $\phi\in \PSH(L)$ and let $\widehat{\phi}_t$ be the Legendre transform associated to the test curve $\psi_{\lambda}$.  Then
\[ \frac{d\widehat{\phi}_t}{dt}\bigg\vert_{t=0^+}  = \,H_{\phi}\]
\end{theorem}

The proof will rely on an elementary lemma from convex geometry.

\begin{lemma}\label{lem:convex}
  Let $u=u_t$ be a real valued convex function in the one variable $t$ such that $u_t=u_0$ for $t\le 0$ and set $v_{\lambda} = \inf_t \{ u_t-\lambda t\}$. Then 
\[\frac{du}{dt}\bigg\vert_{t=0^+} = \sup \{ \lambda : v_{\lambda} = u_{0}\}. \]
\end{lemma}
\begin{proof}
Observe that by convexity and the assumption that $u_t$ is constant for $t<0$ we have $u_t$ is increasing in $t$.  Moreover $v_{\lambda} \le u_0$ for all $\lambda$ and so $v_{0} = \inf_t u_t = u_0$.  Thus the set $S=\{\lambda : v_{\lambda} = u_{0}\}$ is a non-empty interval in $\mathbb R$.  Setting $w(\epsilon) := \epsilon^{-1} (u_{\epsilon} - u_{0})$ for $\epsilon>0$,  the convexity of $u_t$ implies $w$ is non-decreasing and we have to show
\[\lim_{\epsilon\to 0^+} w(\epsilon) = \sup S.\]
First suppose $\lambda\in S$.  Then $u_t-\lambda t\ge u_{0}$ for all $t$, and so $w(\epsilon)\ge \lambda$ for all $\epsilon>0$, and hence $\lim_{\epsilon\to 0^+} w(\epsilon)\ge \sup S$.  In the other direction, suppose $\lambda>\sup S$ and pick some $\lambda'\in S$ with $\lambda'<\lambda$.  Then as $\lambda\notin S$ there is a $t$ so that $u_{t} - t\lambda<u_{0}$.  But as $\lambda'\in S$ we certainly have $u_{0} = v_{\lambda'}\le u_t - \lambda' t$ and putting these together shows $t>0$.  Thus $w(t)<\lambda$ and so $\lim_{\epsilon\to 0+} w(\epsilon)\le \lambda$ by monotonicity of $w$.  Since $S$ is an interval and this holds for all $\lambda>\sup S$ we conclude $\lim_{\epsilon\to 0+} w(\epsilon)\le \sup S$ as required.
\end{proof}

\begin{proof}[Proof of Theorem \ref{thm:rightderivative}]

Define 
\[ \gamma_{\lambda} = \inf_t \{\widehat{\phi}_t-\lambda t\}  \]
By the Kiselman minimum principle $\gamma_{\lambda}\in \PSH(L)$.   We claim that
\begin{equation}
 \gamma_{\lambda} = \phi_{\lambda}.\label{eq:kiselmanidentity}
\end{equation}
Assuming this, the result we want follows directly from the previous Lemma \eqref{lem:convex} as
\[\frac{d\widehat{\phi}_t}{dt}\bigg\vert_{t=0^+} = \sup \{ \lambda : \gamma_{\lambda} = \widehat{\phi}_{0} \} = \sup \{ \lambda : \phi_{\lambda}  = \phi\} = H_{\phi}.\]

Thus is remains to prove \eqref{eq:kiselmanidentity}.   Note that $\widehat{\phi}_t(x) = \sup_{\lambda} \{\phi_{\lambda}(x) + \lambda t\}$ for almost every $x$ (the point being that the Legendre transform also requires us to take the upper semicontinuous regularisation).   Thus for such $x$ an elementary argument (essentially the involution of the Legendre transform) yields $\gamma_{\lambda}(x) = \phi_{\lambda}(x)$.  Thus $\gamma_{\lambda}$ and $\phi_{\lambda}$ are two plurisubharmonic functions that agree almost everywhere, and hence are identically equal.

\end{proof}

\section{Divisorial Exhaustion maps}\label{sec:divisorial}

We now restrict  to the special case of the exhaustion map associated to a divisor $D$.  Let $\lambda_{\max} = \sup\{\lambda : L-\lambda D \text{ is big}\}$.  We assume that $D$ does not intersect $\mathbb B_+(L-\lambda D)$ for all $\lambda\in (0,\lambda_{\max})$.   Define $\psi_{\lambda} = \lambda \ln |s_D|^2$ for $\lambda\in (0,\lambda_{\max})$ where $s_D$ is the defining section of $D$.   

As above we set $\phi_{\lambda} = \phi_{[\psi_{\lambda}]}$ and $D_{\lambda} = \{\phi_{\lambda} = \phi\}$  (and the usual convention that $\phi_{\lambda} = \phi$ for $\lambda<0$ so $D_{\lambda} =X$, and for $\lambda>\lambda_{\max}$ we set $\phi_{\lambda} \equiv -\infty$ so $D_{\lambda}$ is empty.).  We write the associated exhaustion function as $H$ or $H_D$.

Now the volume of the equilibrium sets measures the rate of growth of the subspace of sections of $L^k$ contained in the relevant multiplier ideal $\mathcal I(k\psi_{\lambda}) = \mathcal I(k\ln |s_D|^2) = O(-kD)$ (Corollary \ref{cor:equilibriumvolume}).  If $\lambda\in (0,\lambda_{max})$ then $L-\lambda D$ is big, and thus
\[h^0(k(xL-\lambda D)) = \vol(L-\lambda D) k^n + O(k^{n-1})\]
where $\vol(L-\lambda D) := \frac{1}{n!}\int_X (c_1(L-\lambda D))^n$.  Thus we conclude
\begin{equation}
  \label{eq:volume}
\vol(D_{\lambda},MA(\phi)) := \int_{D_{\lambda}} MA(\phi) = \vol(L-\lambda D) \quad\text{ for all } \lambda.  
\end{equation}

Our goal is to analyse this volume change in terms of the sections of $kL$ that vanish to a certain order along $D$.   To do so it is natural to express our results in terms of the Okounkov body  whose construction we briefly recall now (and refer the reader to \cite{Lazarsfeld, RossNystrom, Kaveh} for details).

Let $X \supset Y_{n-1} \supset \cdots \supset Y_1$ be a flag of smooth subvarieties of $X$ with $\dim Y_i =i$ and $n=\dim X$.   Starting with $Y_{n-1}$ we have a valuation 
\[\nu_1\colon H^0(X,L^k) \to \mathbb Z \quad\text{given by }\quad \nu_1(s) = \ord_{Y_{n-1}}(s),\]
where $\ord_{Y_{n-1}}$ is the order of vanishing along $Y_{n-1}$. If $t$ denotes the defining equation for $Y_{n-1}$, then by definition $\bar{s} := s t^{-\nu_1(s)}$ restricts to a non-trivial section of $L|_{Y_{n-1}}$, and thus we have a second valuation  \[\nu_2(s) = \ord_{Y_{n-2}}(\bar{s}|_{Y_{n-2}}).\]  Proceeding in this way gives a map
\[\nu\colon H^0(L^k) \to \mathbb Z^n \quad\text{given by }\quad \nu(s) = (\nu_1(s),\ldots,\nu_n(s)).\]
We set $\Delta_k = k^{-1} \im(\nu\colon H^0(L^{\otimes k})\to \mathbb Z^n)$ and the Okounkov body is defined to be
 \[ \Delta= \Delta(X,L) =\overline{\bigcup_k  \operatorname{Convex}(\Delta_k)}\]
where $\operatorname{Convex}$ denotes the taking the convex hull and the bar denotes topological closure.

When $L$ is big, the volume of $\Delta$ taken with respect to the Lebesgue measure is precisely the volume of $L$  taken with respect to the line bundle $L$.  This fundamental property lies at the cornerstone of the work of Lazarsfeld-Musta\c{t}\u{a} who use the Okounkov body to study the volume functional on the space of big line bundles, and of Kaveh-Khovanskii who give applications by considering even more general valuations.  \medskip

Suppose now our flag of smooth subvarieties whose divisorial part is given by $D$ (i.e.\ $Y_{n-1}=D$ in the notation above).  

  \begin{theorem}\label{thm:pushforwarddivisor}
The pushforward of the volume form $MA(\phi)$ under the exhaustion function $H_D$ is given by
    \[H_{D*} (MA(\phi)) = p_{1*} (d\sigma|_{\Delta(X,L)})\]
    where $d\sigma$ denotes the Lebesgue measure on $\mathbb R^n$ and $p_1\colon \mathbb R^n\to \mathbb R$ is the projection to the first coordinate.
  \end{theorem}

  \begin{proof}[Proof of \ref{thm:pushforwarddivisor}]
As we will see, this result follows rather easily from our knowledge of the volume of the equilibrium sets \eqref{eq:volume}.    For
    $\lambda\in \mathbb Q$ let $U_{\lambda} = (\lambda,\infty)$. First
    observe that if $\lambda<0$ (resp.\ $\lambda>\lambda_{\max}$) then
    $p_1^{-1}(U_{\lambda}) = \Delta(X,L)$ (resp.\ is empty) and
    $H^{-1}(U_{\lambda})=X$ (resp.\ is empty) and so both
    measures in question are concentrated on the interval
    $[0,\lambda_{\max}]$.  So now let $\lambda\in
    (0,\lambda_{\max})\cap \mathbb Q$.  Then by construction
    $D_{\lambda}\subset H^{-1} (U_\lambda)$ and so
    \[\vol(U_{\lambda}, H_* MA(\phi)) \ge \vol(D_{\lambda}, MA(\phi))
    = \vol(L-\lambda D).\] On the other hand $p_1^{-1}(U_{\lambda})
    \cap \Delta(X,L))$ is (a translate of) the Okounkov body
    of $X$ taken with respect to $L-\lambda D$ \cite[4.24]{Lazarsfeld}. Thus
    \[\vol(U_\lambda,H_*MA(\phi)) \ge \vol(L-\lambda D)=\vol(
    p_1^{-1}(U_{\lambda}),d\sigma).\] Since this holds for rational
    $\lambda$, by continuity it holds for all $\lambda\in
    (0,\lambda_{\max}$).  But the total mass of the two measures in
    question is equal to $\vol(L)$, and thus since they are both
    positive measure they must be equal.
  \end{proof}

For our second result along these lines notice that by construction $|\Delta_k| = h^0(L^k)$.   Moreover the points in $\Delta_k$ determine a filtration of $H^0(L^k)$ obtained by the valuation, namely for $\alpha\in \Delta_k$ set $\mathcal F_{\alpha} = \{ s\in H^0(L^k) : \nu(s) \ge k\alpha\}$ where the inequality is taken in the lexicographic order.    Thus using the $L^2$-inner product on $H^0(L^k)$ we see there is a unique $L^2$-orthonormal basis $\{s_{\alpha}\}$ for $\alpha\in \Delta_k$ for $H^0(L^k)$ with the property 
\[ \nu(s_{\alpha}) = \alpha k \quad \text{ for } \quad \alpha\in \Delta_k\]
\begin{remark}
In the toric case the Okounkov body is nothing other than the usual Delzant polytope and $k\Delta_k$ is precisely the integral points in $k\Delta$ and one can pick torus invariant sections to achieve the same result.  Thus what we are doing here can be thought of as a generalisation of the usual toric picture in which the torus action has be replaced with the data of a divisor $D$ in $X$ and a hermitian metric $\phi$ on $L$.  
\end{remark}

\begin{theorem}\label{theorem:gradientmapfromsections2}
  We have
\[H_D = \limsup_{k\to \infty} \frac{\sum_{\alpha} \alpha_1 |s_{\alpha}|_\phi^2}{\sum_{\alpha} k|s_{\alpha}|_\phi^2}\quad \text{where } \alpha_1 = p_1(\alpha).\]
almost everywhere on $X$
\end{theorem}

The proof of this will uses the connection between the partial Bergman function and the maximal envelopes.   For fixed rational $\lambda$, we consider the partial Bergman kernel
\[ B_{\lambda, k} = B(\lambda\psi,\phi)= \sum_{\beta} |s_{\beta}|_{\phi}^2 \quad \text{for }k\lambda \in \mathbb N\]
where $\{s_{\beta}\}$ is any $L^2$-orthonormal basis for the subspace $H^0(\mathcal I_D^{\lambda k}L^k)$.  
Since this definition is independent of basis chosen, in terms of the notation above it is thus given by
\[B_{\lambda,k} = \sum_{\alpha_1\ge \lambda} |s_{\alpha}|_{\phi}^2.\]
(the sum being understood as over all $\alpha\in \Delta_k$ whose first coordinate is at least $\lambda$).

\begin{proof}[Proof of \ref{theorem:gradientmapfromsections}]
  By the standard asymptotic of the Bergman function, 
  $k^{-n}\sum_{\alpha} |s_{\alpha}|_{\phi}^2$ tends to $1$ uniformly
  on $X$ as $k$ tends to infinity.  Thus if we let
  \begin{eqnarray*}
f_k &=& k^{-n} \sum_{\alpha} \alpha_1 |s_{\alpha}|_\phi^2
\end{eqnarray*}
and
\[ f:=\limsup_k f_k\]
it becomes sufficient to prove that $f= H_D$
almost everywhere on $X$.      To this end, fix $x\in X$ and some rational $\lambda'>H_D(x)$.   Then if $\lambda'\ge \lambda$ we have $B_{\lambda,k}\le C_k e^{\phi(x) - \phi_{\lambda}(x)} \le C_k e^{\phi(x) - \phi_{\lambda'}(x)}$.  Thus
\begin{eqnarray*}
f_k(x) &=&k^{-n}\sum_{\lambda\in p_1(\Delta_k)} B_{\lambda,k}(x)\\
&=& k^{-n} \sum_{\lambda\le \lambda'} B_{\lambda,k}(x) + k^{-n} \sum_{\lambda\ge \lambda'} B_{\lambda,k}(x)\\
&=& \le C_k \lambda' + Ck^n e^{k(\phi(x) - \phi_{\lambda'}(x))}.
\end{eqnarray*}
where $C_k$ is a sequence of constants that tends to $1$ as $k$ tends to infinity.  Now since $\lambda'>H_D(x)$ we have $\phi_{\lambda'}(x)<\phi(x)$.  Thus taking the limsup yields $f(x)\le \lambda'$ and letting $\lambda'$ tend to $H_D(x)$ we deduce $f(x)\le H_D(x)$ for all $x\in X$.

Now as each $s_{\alpha}$ has unit $L^2$-norm we clearly have
\[\int_X f_k MA(\phi) = \sum_{\alpha\in \Delta_k} \alpha_1\to \int_{\Delta} x_1 d\sigma\]
where $d\sigma$ is the Lebesgue measure and $x_1$ is the first
coordinate.  But from the pushforward property (Theorem
\ref{thm:pushforwarddivisor}) this last integral is equal to $\int_X
H_{D} MA(\phi)$.  Thus from Fatou's Lemma,
\[\int_X f MA(\phi) = \int_X \limsup f_k MA(\phi) \ge \lim_k \int_X f_k MA(\phi) = \int_X H_D MA(\phi).\]
Hence we must in fact have $\int_X f MA(\phi) = \int_X H_D MA(\phi)$, and
thus $f=H_D$ almost everywhere on $X$ as required.
\end{proof}

\noindent {\sc Julius Ross,  University of Cambridge, UK. \\j.ross@dpmms.cam.ac.uk}\vspace{2mm}\\ 
\noindent{\sc David Witt Nystrom,  University of Gothenburg, Sweden. \\\quad wittnyst@chalmers.se}

\end{document}